\documentclass[a4paper,12pt]{article}
\usepackage{graphicx} 
\usepackage{listings}
\usepackage[english]{babel}
\usepackage{geometry}
\usepackage{enumitem} 

\usepackage[utf8]{inputenc}

\usepackage{amsmath,amsthm,amssymb,mathrsfs,titlesec,fancyhdr}
\usepackage{float} 
\usepackage{textcomp}
\usepackage{url}

\usepackage{xcolor, colortbl}
\definecolor{verdet}{rgb}{0.039, .607, .211}

\def\Xint#1{\mathchoice
{\XXint\displaystyle\textstyle{#1}}%
{\XXint\textstyle\scriptstyle{#1}}%
{\XXint\scriptstyle\scriptscriptstyle{#1}}%
{\XXint\scriptscriptstyle\scriptscriptstyle{#1}}%
\!\int}
\def\XXint#1#2#3{{\setbox0=\hbox{$#1{#2#3}{\int}$ }
\vcenter{\hbox{$#2#3$ }}\kern-.6\wd0}}
\def\dashint{\Xint-}

\theoremstyle{plain}
    \newtheorem{lema}{Lemma}[section]

    \newtheorem{coro}{Corollary}[section]
    \newtheorem{teo}{Theorem}[section]
    
\theoremstyle{definition}

\newcommand{\N}{\mathbb{N}}

\newcommand{\R}{\mathbb{R}}

\newcommand{\Capp}{\operatorname{Cap}}

\newcommand{\supp}{\operatorname{supp}}

\providecommand{\norm}[1]{\lVert#1\rVert}

\title{Failure of almost monotonicity of harmonic measure density ratios at
points of vanishing codimension-one density}

\author{Luis Lloret and Xavier Tolsa\thanks{L.L. and X.T. were supported by the European Research Council (ERC) under the European
Union’s Horizon 2020 research and innovation programme (grant agreement 101018680). Also partially
supported by MICIU (Spain) under the grant PID2024-160507NB-I00.}}

\begin{document}
\maketitle 
\begin{abstract}
In this paper, we study the behavior of the density ratios of harmonic measure at points with vanishing density. Given an arbitrary open set $\Omega\subset\R^{n+1}$ with harmonic measure $\omega$, we show that at $\omega$-almost every point $x\in\partial\Omega$ where $\liminf_{r\to0}\frac{\omega(B(x,r))}{r^n}=0$, the density ratio $\frac{\omega(B(x,r))}{r^n}$ is not almost monotone with respect to the radius $r$, and therefore exhibits arbitrarily large oscillations at small scales.
\end{abstract}

\maketitle

\section{Introduction}
This paper is devoted to the study of the metric and geometric properties of harmonic measure, a central object in potential theory and in the analysis of the Dirichlet problem. In the planar case, in 1988 Jones and Wolff \cite{JW} proved that for an arbitrary open set, the dimension of harmonic measure is at most one. That is, there exists a subset of the boundary with Hausdorff dimension at most one having full harmonic measure. Later on, this result was sharpened by Wolff \cite{W2}, who proved the existence of a set with $\sigma$-finite one-dimensional Hausdorff measure carrying full harmonic measure. 

The results of Jones-Wolff and Wolff play a fundamental role in the study of the dimension of harmonic measure in $\mathbb{R}^2$. 
For example, as a consequence of Wolff's theorem, the set of points with vanishing one-dimensional density has zero harmonic measure, that is
\begin{equation*}
\omega\left(\left\{x\in\mathbb{R}^2:\lim_{r\to0}\frac{\omega(B(x,r))}{r}=0\right\}\right)=0.
\end{equation*}
See Section \ref{sec4} below for the detailed argument.

The situation changes in higher dimensions. Bourgain \cite{B} proved that there exists a constant $c_{n+1}>0$, depending only on the dimension, such that
\[
\dim_{\mathcal{H}}\omega\leq n+1-c_{n+1}.
\]

On the other hand, in $\mathbb{R}^3$, Wolff \cite{W} constructed a domain whose harmonic measure is supported on a set of Hausdorff dimension strictly larger than $2$, showing that necessarily $c_3<1$ in this case. In particular, denoting 
\begin{equation}\label{eq:defi_Z}
    \Theta_{\omega}^n(x,r)=\frac{\omega(B(x,r))}{r^n},\qquad Z=\left\{x\in\mathbb{R}^{n+1}:\lim_{r\to0}\Theta_{\omega}^n(x,r)=0\right\},
\end{equation}
Wolff's example satisfies $\omega(Z)>0$. Despite their importance, the metric and geometric properties of such exceptional sets have received little attention, largely because very few explicit examples are known.

The main purpose of this paper is to establish the following result:

\begin{teo}\label{teo:main_result}
Let $\Omega\subset\mathbb{R}^{n+1}$ be an open set with compact boundary, let $p\in\Omega$, and let $\omega_\Omega^p$ be the harmonic measure in $\Omega$ with pole at $p$ and consider the set $Z$ defined in \eqref{eq:defi_Z} (with $\omega=\omega^p_\Omega$). Then
\begin{equation*}
\limsup_{R\to0}\sup_{0<r<R}\frac{\Theta_{\omega_\Omega^p}^n(x,r)}{\Theta_{\omega_\Omega^p}^n(x,R)}=\infty\qquad\text{for }\omega_\Omega^p\text{-almost every }x\in Z.
\end{equation*}
\end{teo}

Recall that given an open interval $I\subset\mathbb{R}$, a function $g:I\to\mathbb{R}$ is said to be \emph{almost increasing} if there exists a constant $C>0$ such that
\begin{equation*}
g(r)\le Cg(s)\qquad\text{whenever }r\leq s,\qquad r,s\in I.
\end{equation*}
Similarly, $g$ is \emph{almost decreasing} if
\begin{equation*}
g(s)\leq Cg(r)\qquad\text{whenever }r\leq s,\qquad r,s\in I,
\end{equation*}
and it is called \emph{almost monotone} if it is almost increasing or almost decreasing.

Theorem \ref{teo:main_result} shows that, for $\omega_\Omega^p$-almost every $x\in Z$, the density ratio function
\begin{equation*}
r\mapsto\Theta_{\omega_\Omega^p}^n(x,r)
\end{equation*}
fails to be almost monotone in any interval $(0,r_0)$ for any $r_0>0$. In particular, although the density tends to zero as $r\to0$, the density ratios exhibit oscillations of arbitrarily large amplitude at sufficiently small scales. We emphasize that this phenomenon cannot occur when harmonic measure is mutually absolutely continuous with respect to the $n$-dimensional Hausdorff measure. Indeed, if $E\subset \partial\Omega$ is such that $\omega|_E$ and the Hausdorff measure $\mathcal H^n|_E$ are mutually absolutely continuous, then $E$ is $n$-rectifiable by \cite{T1}. By the Radon–Nikodym theorem together with the density theorem for rectifiable sets, this implies that $ \lim_{r\to0}\Theta_{\omega}^n(x,r)$ exists and it is positive $\omega$-a.e.\ in $E$ (see \cite[Chapter 17]{MAT}). So the density ratio $\Theta_{\omega}^n(x,r)$ is both almost increasing and almost decreasing near $0$.
Thus, Theorem 1.1 can be interpreted as a converse, in a qualitative sense, to the asymptotic regularity enjoyed by harmonic measure on rectifiable pieces of the boundary.

The main new ingredient to prove Theorem \ref{teo:main_result} is a rectifiability criterion due to the second author of this paper \cite[Th.~1.5]{T}, which provides sufficient conditions for a Radon measure to charge a rectifiable set. Following \cite[Th.~3.3]{AMT}, in Theorem \ref{teo:main1} below, we first prove a variant of that rectifiability criterion which is more suitable for our purposes. This result constitutes the main technical tool in the proof of Theorem \ref{teo:main_result}.

The argument proceeds by contradiction. We assume that the density ratios are almost monotone on a subset of $Z$ of positive harmonic measure, and then we show that this assumption yields a ball for which the harmonic measure satisfies all the hypotheses of Theorem \ref{teo:main1}. Consequently, one obtains an $n$-rectifiable set carrying positive harmonic measure contained in $Z$. This contradicts the fact that harmonic measure vanishes on every $n$-rectifiable subset of $Z$, thereby proving Theorem \ref{teo:main_result}.

From Theorem \ref{teo:main_result} we get the following apparently stronger result.

\begin{coro}\label{coro1}
Let $\Omega\subset\mathbb{R}^{n+1}$ be an open set with compact boundary, let $p\in\Omega$, and let $\omega_\Omega^p$ be the harmonic measure in $\Omega$ with pole at $p$ and consider the set 
\begin{equation*}
Z_*=\left\{x\in\mathbb{R}^{n+1}:\liminf_{r\to0}\Theta_{\omega^p_\Omega}^n(x,r)=0\right\}.
\end{equation*}
 Then
\begin{equation*}
\limsup_{R\to0}\sup_{0<r<R}\frac{\Theta_{\omega_\Omega^p}^n(x,r)}{\Theta_{\omega_\Omega^p}^n(x,R)}=\infty\qquad\text{for }\omega_\Omega^p\text{-almost every }x\in Z_*.
\end{equation*}
\end{coro}

Observe that the assumption that $\lim_{r\to0}\Theta_{\omega^p}^n(x,r)=0$ in $Z$ has been replaced by the weaker condition asserting that $\liminf_{r\to0}\Theta_{\omega^p}^n(x,r)=0$ in $Z_*$.

The paper is organized as follows. In Section~2 we show how one deduces Corollary \ref{coro1} from the main Theorem \ref{teo:main_result}. We also introduce some notation, recall the necessary background material, and establish several auxiliary lemmas to prove the main theorem. 
Section~3 is devoted to the proof of Theorem \ref{teo:main1}. Roughly speaking,  that theorem asserts the following: suppose that we have a measure $\mu$ and a ball $B_0\subset\mathbb{R}^{n+1}$ for which there exists a large subset $G\subset3 B_0$ satisfying:
\begin{enumerate}[label=(\roman*)]
    \item Suitable polynomial growth conditions of $\mu$ in $G$ and boundedness of its Riesz transform in terms of the density ratio $\Theta_\mu(B_0)$.
    \item  Existence of a ball $B_1$ contained in $B_0$, sufficiently small relative to $B_0$, with  $\Theta_\mu(B_1)\geq c\,\Theta_\mu(B_0)$, for some fixed $c>0$.
    \item Small $L^2(\mu)$ oscillation of the Riesz transform $\mathcal{R}_\mu$ in $G$.
\end{enumerate}
Then a piece of $B_0\cap G$ is contained in an $n$-rectifiable subset.

In Section~4 we reduce the proof of Theorem \ref{teo:main_result} to the case where $\Omega\subset\mathbb{R}^{n+1}$ is Wiener regular for $n\geq2$. Finally, Section~5 contains the proof of the main theorem. There we verify the hypotheses of Theorem \ref{teo:main1} for $\mu$ being the harmonic measure and a ball $B_0$ containing a subset of $Z$ of very large harmonic measure. The condition (i) above is established using the almost monotonicity assumption together with the connection between the Riesz transform and the Green function. To verify the second condition, we develop a touching-point argument after modifying the boundary of the domain by means of suitable ``good'' cubes from the David--Mattila dyadic lattice. This is the  most delicate part of the whole argument. The condition (iii) above is proved by arguments similar to those used for the first one, using again the connection of the Riesz transform with the Green function. The final  contradiction obtained at the end of the section completes the proof of Theorem \ref{teo:main_result}.

\section{Preliminaries}
    The notation $a\lesssim b$ means that there exists some constant $C$ such that $a\leq C b$. Also, $a\approx b$ means $a\lesssim b\lesssim a$. If one wants to specify the dependence of $C$ on some parameter $\lambda$, then one writes $a\lesssim_\lambda b$, and $a\approx_\lambda b$ respectively. On the other hand, if $B=B(x,r)$ is an open ball and $\lambda>0$, then $\lambda B:=B(x,\lambda r)$.

\subsection{Proof of Corollary \ref{coro1} using Theorem \ref{teo:main_result}}

Let $\Omega\subset\R^{n+1}$ and $Z$, $Z_*$ be as in Theorem \ref{teo:main_result} and Corollary
\ref{coro1}. By Theorem \ref{teo:main_result}, the set
\begin{equation*}
    \tilde Z := \Big\{x\in Z\cap\supp\omega^p: 
\limsup_{R\to0}\sup_{0<r<R}\frac{\Theta_{\omega_\Omega^p}^n(x,r)}{\Theta_{\omega_\Omega^p}^n(x,R)}<\infty\Big\}
\end{equation*}
satisfies $\omega^p(\tilde Z)=0$. 
We claim that $\tilde{Z}$ coincides with the set
\begin{equation*}
\tilde Z_* := 
\Big\{x\in Z_*\cap\supp\omega^p: 
\limsup_{R\to0}\sup_{0<r<R}\frac{\Theta_{\omega_\Omega^p}^n(x,r)}{\Theta_{\omega_\Omega^p}^n(x,R)}<\infty\Big\},
\end{equation*}
and so $\omega^p(\tilde Z_*)=0$, which proves the corollary.

We now prove the claim. Indeed, obviously we have $\tilde Z\subset \tilde Z_*$. For the converse inclusion, consider $x\in \tilde Z_*$ and
denote 
\begin{equation*}
c_x=\limsup_{R\to0}\sup_{0<r<R}\frac{\Theta_{\omega_\Omega^p}^n(x,r)}{\Theta_{\omega_\Omega^p}^n(x,R)}.
\end{equation*}Let $\{r_k\}_{k\in\mathbb{N}}$ be a decreasing sequence of radii such that $\lim_{k\to\infty}\Theta_{\omega^p}^n(x,r_k)=0$. For a given $\varepsilon>0$, let $k_0$ be such that, for $k\geq k_0$,
\begin{equation*}
\Theta_{\omega_\Omega^p}^n(x,r_k)\leq \varepsilon
\end{equation*}
and
\begin{equation*}
\sup_{0<r<r_k}\frac{\Theta_{\omega_\Omega^p}^n(x,r)}{\Theta_{\omega_\Omega^p}^n(x,r_k)}\le 2c_x.
\end{equation*}
Then, for $0<r\leq r_{k_0}$, it holds that
\begin{equation*}
\Theta_{\omega_\Omega^p}^n(x,r)\leq 2c_x\Theta_{\omega_\Omega^p}^n(x,r_{k_0})\leq 2c_x\varepsilon,
\end{equation*}
which proves the claim.

\subsection{Doubling conditions, Riesz transforms, and capacity}
Given a ball $B=B(x,r)\subset\mathbb{R}^{n+1}$, a Radon measure $\mu$ in $\mathbb{R}^{n+1}$, and $\gamma>0$, we denote
\begin{equation*}
    \Theta_{\mu}^n(x,r)=\Theta_{\mu}^n(B)=\dfrac{\mu(B)}{r(B)^n},\quad P_{\gamma,\mu}(B)=\sum_{j\geq0}2^{-j\gamma}\Theta^n_{\mu}(2^jB).
\end{equation*}
We call $\Theta^n_{\mu}(B)$ the $n$-dimensional density ratio of $\mu$ on $B$ and $P_{\gamma,\mu}(B)$ the $(n,\gamma)$-Poisson density of $\mu$ on $B$. Observe that $P_{\gamma,\mu}(B)\geq \Theta^n_{\mu}(B)$.
 We consider the maximal operator
\begin{equation*}
    \mathcal{M}_n\mu(x)=\sup_{r>0}\Theta_\mu^n(x,r) = \sup_{r>0}\frac{\mu(B(x,r))}{r^n}.
\end{equation*}
Given a constant $C>0$, we say that a ball $B$ is $(C,P_{\gamma,\mu})$-doubling if $P_{\gamma,\mu}(B)\leq C\Theta^n_{\mu}(B)$,
and $C$-doubling if $\Theta^n_{\mu}(2B)\leq C\Theta^n_{\mu}(B)$.

For a Radon measure $\mu$ in $\mathbb{R}^{n+1}$, we consider the Riesz transform 
\begin{equation*}
    \mathcal{R}\mu(x)=\int\dfrac{x-y}{|x-y|^{n+1}}\,d\mu(y),
\end{equation*}
whenever this vectorial integral makes sense. For $\varepsilon>0$, we also consider the truncated version
\begin{equation*}
    \mathcal{R}_\varepsilon\mu(x)=\int_{|x-y|>\varepsilon}\dfrac{x-y}{|x-y|^{n+1}}\,d\mu(y).
\end{equation*}
We can also define the maximal version as
\begin{equation*}
    \mathcal{R}_\ast\mu(x)=\sup_{\varepsilon>0}|\mathcal{R}_\varepsilon\mu(x)|.
\end{equation*}

For $f\in L^1_{loc}(\mu)$ and $A\subset\mathbb{R}^{n+1}$, we write
\begin{equation*}
    m_{\mu, A}(f)=\dfrac{1}{\mu(A)}\int_{A} f\,d\mu=\dashint_Af\,d\mu.
\end{equation*}
On the other hand, for $n\geq2$, the fundamental solution of the Laplacian equals
\begin{equation*}
    \mathcal{E}(x)=\dfrac{1}{\kappa_{n+1}|x|^{n-1}},
\end{equation*}
where $\kappa_{n+1}=(n-1)(n+1)\mathcal{L}^{n+1}(B(0,1))$ and $\mathcal{L}^{n+1}$ is the usual Lebesgue measure in $\mathbb{R}^{n+1}$. We can now define the Newtonian potential of $\mu$ as
\begin{equation*}
    U_\mu(x)=\int\mathcal{E}(x-y)\,d\mu(y).
\end{equation*}
Given a compact set $E\subset\mathbb{R}^{n+1}$, its Newtonian capacity is
\begin{equation*}
    \Capp(E)=\sup_{\mu\in M_+(E)}\left\{\mu(E)~:~U_\mu(y)\leq 1~\text{for all }y\in\mathbb{R}^{n+1}\right\},
\end{equation*}
where $M_+(E)$ is the set of non-negative Radon measures supported on $E$ (see \cite[Chapter II]{La} or \cite[Chapter 6]{HM}). Furthermore, an open set $\Omega\subset\mathbb{R}^{n+1}$ satisfies the capacity density condition, usually abbreviated by CDC, if there exists a constant $c>0$ such that
\begin{equation*}
    \Capp(\overline{B(x,r)}\setminus\Omega)\geq cr^{n-1}
\end{equation*}
for all $0<r<\operatorname{diam}(\partial\Omega)$ and all $x\in\partial\Omega$. 

\subsection{Harmonic measure}

Most of the results of this section can be found in \cite{HM}.

For an open set $\Omega\subset \mathbb{R}^{n+1}$ with compact boundary and $f\in\mathcal{C}(\partial\Omega)$,
the Dirichlet problem consists in finding $F\in\mathcal{C}^2(\Omega)\cap \mathcal{C}(\overline{\Omega})$ such that
\begin{equation*}
    \left\{\begin{aligned}
\Delta F=0 &\quad\text{in }\Omega,\\
F=f&\quad\text{on }\partial\Omega.
\end{aligned}\right.
\end{equation*}
When $\Omega$ is unbounded and $n\geq 2$, we also require 
\begin{equation*}
    \lim_{x\in\Omega,~|x|\to\infty}F(x)=0.
\end{equation*}
One says that $\Omega$ is Wiener regular if this problem has a solution for all $f\in\mathcal{C}(\partial\Omega)$. 

For an arbitrary open set $\Omega$ with compact boundary which may be non-Wiener regular, one can consider the so called ``Perron solution" of the Dirichlet problem, which is obtained by the Perron method. This solution coincides with the classical one when $\Omega$ is Wiener regular. In any case, for all $p\in\Omega$, by the Riesz representation theorem, there exists a unique Radon measure $\omega^p_\Omega$ such that the Perron solution of the Dirichlet problem satisfies the following equality:
\begin{equation*}
    F(p)=\int_{\partial\Omega}f(x)\,d\omega^p_\Omega(x).
\end{equation*}
The measure $\omega^p_\Omega$ is called harmonic measure of $\Omega$ with pole $p$. When $\Omega$ is bounded, this is a probability measure. Moreover, if $A\subset\partial\Omega$ is a Borel set, $x\mapsto\omega^x_\Omega(A)$ is a harmonic function in $\Omega$.

On the other hand, if $\Omega$ is Wiener regular, we can define the \textit{Green function} for $\Omega$  with pole at $x$, $G^x_\Omega:\Omega\setminus \{x\}\to\R$, as
\begin{equation}\label{eq:Green_func}
        G_\Omega^x(y)=\begin{cases}
    \mathcal{E}(x-y)-\int_{\partial\Omega}\mathcal{E}(x-z)d\omega_{\Omega}^y(z) &\quad\text{in }\Omega\setminus\{x\},\\
    0&\quad\text{in }\mathbb{R}^{n+1}\setminus\Omega.
\end{cases}
\end{equation}
In the case $n\geq2$, the identity also holds if $\Omega$ is an unbounded open with compact boundary. Further, by the same arguments as in \cite[L.~8.4]{HM}, for arbitrary Wiener regular open sets with compact boundaries, one can show the identity 
\begin{equation*}
G_\Omega^x(y)=\mathcal{E}(x-y)-\int_{\partial\Omega}\mathcal{E}(x-z)d\omega_{\Omega}^y(z)\quad \text{for $\mathcal{L}^{n+1}$-a.e. $x\in\mathbb{R}^{n+1}$.}
\end{equation*}

\begin{lema}\label{lema:formula_Green_medida_armonica}\cite[L.~8.13]{HM}
    Let $\Omega\subset\mathbb{R}^{n+1}$ be an open Wiener regular set with compact boundary. For all $x\in\Omega$ and all $\varphi\in\mathcal{C}^\infty_c(\mathbb{R}^{n+1})$, we have
    \begin{equation*}
        \int_{\partial\Omega}\varphi(y)\,d\omega^x(y)-\varphi(x)=\int_\Omega\Delta\varphi(y)G_\Omega^x(y)\,dy.
    \end{equation*}
\end{lema}
\begin{lema}\label{lema:cota_superior_medida_armonica_bolas}\cite[L.~8.16]{HM}
    Let $n\geq2$ and $\Omega\subset\mathbb{R}^{n+1}$ be an open Wiener regular set with compact boundary. For any closed ball $\overline{B}$, it holds that
    \begin{equation*}
        \omega^x(\overline{B})\geq c(n)\dfrac{\Capp(\frac{1}{4}\overline{B}\setminus\Omega)}{r(B)^{n-1}}\quad\text{for all }x\in\frac{1}{4}\overline{B}\cap\Omega,
    \end{equation*}
    with $c(n)>0$.
\end{lema}

\begin{lema}\label{lema:Cap_8B}\cite[L.~8.18]{HM}
Let $n\geq2$ and $\Omega\subset\mathbb{R}^{n+1}$ be an open Wiener regular set with compact boundary. For any closed ball $\overline{B}$, it holds that
\begin{equation*}
    \Capp(2\overline{B}\setminus \Omega)G_\Omega^x(y)\lesssim_{n}\omega^x(8\overline{B})\quad\text{for all }x\in\Omega\setminus2\overline{B}\text{ and }y\in\overline{B}\cap\Omega.
\end{equation*}
\end{lema}

\begin{lema}\label{lema:bolas_P_doblantes_medida_armonica}\cite[L.~6.1]{AMT}
    There is $\gamma_0\in(0,1)$ so that the following holds. Let $\Omega\subset\mathbb{R}^{n+1}$ be any domain. For all $\gamma>\gamma_0$, there exists some big enough constant $a=a(\gamma,n)$ such that for $\omega^x_\Omega$-a.e. $y\in\mathbb{R}^{n+1}$ there exists a sequence of $(a, P_{\gamma,\omega^x})$-doubling balls $B(y, r_i)$, with $r_i\to0$ as $i\to\infty$. 
\end{lema}

\begin{lema}\label{lema:relacion_contenido_Hausodrf_capacidad}\cite[Th.~6.20]{HM}
    Let $E\subset\mathbb{R}^{n+1}$ be compact and $n-1<s\leq n+1$, then
    \begin{equation*}
        \mathcal{H}^s_{\infty}(E)^{\frac{n-1}{s}}\lesssim_{n,s}\Capp(E)\lesssim_{n}\mathcal{H}_\infty^{n-1}(E),
    \end{equation*}
    where $\mathcal{H}_\infty^d$ denotes the Hausdorff content of dimension $d$.
    \end{lema}

    \begin{lema}\label{lema:Markov_property_harmonic_measure}\cite[Th.~5.54]{HM}
        Let $\Omega,\tilde{\Omega}\subset\mathbb{R}^{n+1}$ be open sets with compact boundaries such that $\Omega\subset\tilde{\Omega}$. Suppose also that $\Omega$ is Wiener regular and that the points from $\partial\Omega \cap\partial\tilde{\Omega}$ are regular for $\tilde{\Omega}$. Denote by $\omega_\Omega$ and $\omega_{\tilde{\Omega}}$ the respective measures for $\Omega$ and $\tilde{\Omega}$. Then, for every $x\in\Omega$ and every Borel set $A\subset\partial\tilde{\Omega}$, it holds
        \begin{equation*}
            \omega_{\tilde{\Omega}}^x(A)=\omega^x_\Omega(A)+\int_{\partial\Omega\setminus\partial\tilde{\Omega}}\omega^y_{\tilde{\Omega}}(A)\,d\omega^x_\Omega(y).
        \end{equation*}
    \end{lema}

    \begin{lema}\label{lema:modificacion_dominios_no_WR}\cite[P.~6.37]{HM}
        Let $\Omega\subset\mathbb{R}^{n+1}$ be open with compact boundary and let $p\in\Omega$. Let $W\subset\partial\Omega$ be the family of irregular points of $\Omega$. For any $\varepsilon>0$, there exists a covering of $W$ by a countable or finite family of closed balls $\{\overline{B_i}\}_{i\in I}$ satisfying the following properties:
        \begin{enumerate}[label=\roman*)]
            \item The balls $\overline{B_i}$ are centered in $\partial\Omega$ and they have bounded overlap.
            \item $\Capp(\bigcup_{i\in I}2\overline{B_i})\leq\varepsilon$.
            \item ${\Omega}_\varepsilon:=\Omega\setminus\bigcup_{i\in I}\overline{B_i}$ is open and it holds that 
            ${\Omega}_\varepsilon=\Omega\setminus\overline{\bigcup_{i\in I}\overline{B_i}}$.
            \item $\partial{\Omega}_\varepsilon\subset\left( \partial\Omega\setminus\bigcup_{i\in I}\overline{B_i} \right)\cup\bigcup_{i\in I}\partial\overline{B_i}$.
            \item ${\Omega}_\varepsilon$ is Wiener regular.
            \item \label{itemvi} For any $x\in{\Omega}_\varepsilon$, $n\geq2$, we have
            \begin{equation*}
                \omega^x_{{\Omega}_\varepsilon}\left( \bigcup_{i\in I}2\overline{B_i} \right)\leq\varepsilon\sup_{y\in\partial{\Omega}_\varepsilon} \mathcal{E}(x-y).
            \end{equation*}
        \end{enumerate}
    \end{lema}

    Observe that, for $x\in\Omega$ and $\varepsilon\in(0,1)$, the term $ \sup_{y\in\partial{\Omega}_\varepsilon} \mathcal{E}(x-y)$ on the right hand side of the inequality in \ref{itemvi} is finite and uniformly bounded on $\varepsilon>0$ because $d(x,\partial\Omega_\varepsilon)\geq d(x,\partial\Omega)>0$ and $\partial\Omega_\varepsilon$ is contained in a fixed compact set independent of $\varepsilon$.

    \begin{lema}\label{lema:aproximacion_medida_con_dominios_WR_modificados}\cite[L.~6.38]{HM}
Let $n\geq2$ and $\Omega\subset\mathbb{R}^{n+1}$ be open with compact boundary and let $p\in\Omega$. For any $\varepsilon>0$, denote by ${\Omega}_\varepsilon$ the Wiener regular set constructed in Lemma \ref{lema:modificacion_dominios_no_WR}. Then, for any Borel set $A\subset\partial\Omega$, it holds that
\begin{equation*}
    \lim_{\varepsilon\to0}\omega_{{\Omega}_\varepsilon}^p(A)=\omega_\Omega^p(A).
\end{equation*}
    \end{lema}

\subsection{David and Mattila lattice}\label{sec2.3}

Here we introduce the dyadic lattice cubes of David-Mattila \cite{DM} associated to a Radon measure $\mu$. However, we will introduce a small variant which will be more appropriate for our purposes. This variant is obtained by a trivial adaptation of the original proof.

\begin{lema}[David, Mattila]

    Let $\mu$ be a compactly supported Radon measure in $\mathbb{R}^{n+1}$. Consider two constants $C_0>1$ and $A_0>5000C_0$ and denote $F=\operatorname{supp}\mu$. Then there exists a sequence of partitions of $F$ into Borel subsets called ``cubes", $Q\in\mathcal{D}_{\mu,k}$ with the following properties:
    \begin{itemize}
        \item For each integer $k\geq0$, $F$ is the disjoint union of the cubes $Q$, $Q\in\mathcal{D}_{\mu,k}$ and if $k<l$, $Q\in\mathcal{D}_{\mu,k}$ and $R\in\mathcal{D}_{\mu,l}$, then either $Q\cap R=\emptyset$ or else $R\subset Q$.
        \item The general position of the cubes $Q$ can be described as follows. For each $k\geq0$ and each cube $Q\in\mathcal{D}_{\mu,k}$, there is a ball $B(Q)=B(x_Q,r(Q))$ such that
        \begin{equation}\label{eq:control_radios_DM}
            x_Q\in F,\quad A_0^{-k}\leq r(Q)\leq C_0A_0^{-k},
        \end{equation}
        \begin{equation}\label{eq:cadena_inclsuion_cubos_bolas_DM}
            F\cap \tfrac{101}{100}B(Q)\subset Q\subset F\cap28B(Q)=F\cap B(x_Q,28r(Q)),
        \end{equation}
        and
        \begin{equation*}
            \text{the balls }~ 5B(Q),~Q\in\mathcal{D}_{\mu,k}~ \text{ are disjoint.}
        \end{equation*}
        \item The cubes $Q\in\mathcal{D}_{\mu,k}$ have small boundaries. That is, for each $Q\in\mathcal{D}_{\mu,k}$ and each integer $l\geq0$, set
        \begin{equation*}
            N_l^{\text{ext}}(Q)=\{ x\in F\setminus Q:d(x,Q)<A_0^{-k-l} \},
        \end{equation*}
        \begin{equation*}
            N_l^{\text{int}}(Q)=\{ x\in Q:d(x,F\setminus Q)<A_0^{-k-l} \},
        \end{equation*}
        and
        \begin{equation*}
            N_l(Q)=N_l^{\text{int}}(Q)\cup N_l^{\text{ext}}(Q).
        \end{equation*}
        Then
        \begin{equation}\label{eq:small_boundary_DM}
            \mu(N_l(Q))\leq (C^{-1}C_0^{-3(n+1)-1}A_0)^{-l}\mu(90B(Q)),
        \end{equation}
        for some universal constant $C$ that only depends on $n+1$.
        \item Denote by $\mathcal{D}^{db}_{\mu,k}$ the family of cubes $Q\in\mathcal{D}_{\mu,k}$ for which
        \begin{equation}\label{eq:bolas_db_800_DM}
            \mu(800B(Q))\leq C_0\mu(B(Q)).
        \end{equation}
        When $Q\in\mathcal{D}_{\mu,k}\setminus\mathcal{D}^{db}_{\mu,k}$, we have $r(Q)=A_0^{-k}$ and
        \begin{equation}\label{eq:des:bolas_no_db_800_DM}
            \mu(800B(Q))\leq C_0^{-l}\mu(800^{l+1}B(Q))~~\text{ for all }l\geq1~~\text{with}~~800^l\leq C_0.
        \end{equation}
    \end{itemize}
    
\end{lema}

Before going further, we  point out that the only changes that we made from the original David and Mattila construction are in \eqref{eq:cadena_inclsuion_cubos_bolas_DM}, \eqref{eq:bolas_db_800_DM} and \eqref{eq:des:bolas_no_db_800_DM}. For \eqref{eq:cadena_inclsuion_cubos_bolas_DM} we write $\frac{101}{100}B(Q)\cap F\subset Q$ instead of $B(Q)\cap F\subset Q$, which follows by sharpening Lemma 3.55 in \cite{DM} without any change in the construction. For \eqref{eq:bolas_db_800_DM} and \eqref{eq:des:bolas_no_db_800_DM} we write $800B(Q)$ instead of $100B(Q)$, which can be obtained by an appropriate change at the beginning of the construction of the lattice and other trivial changes.

We usually use the notation $\mathcal{D}_\mu=\bigcup_{k\geq0} \mathcal{D}_{\mu,k}$, $\mathcal{D}_{\mu}^{db}=\bigcup_{k\geq0}\mathcal{D}^{db}_{\mu,k}$ and, for some $R\in\mathcal{D}_{\mu,k}$ and $l\geq0$,
\begin{equation*}
    \mathcal{D}_{\mu,l}(R)=\{Q\in\mathcal{D}_{\mu,k+l}:Q\subset R \}.
\end{equation*}
Given an $R\in\mathcal{D}_{\mu,k}$, we also refer to $\operatorname{Ch}(R)=\mathcal{D}_{\mu,1}(R)$ as the children of $R$, and vice versa, if $Q\in\operatorname{Ch}(R)$, then we say that $R$ is the parent of $Q$.

We note that the constant $A_0$ can be chosen as
\begin{equation*}
    A_0=C_0^{c_{n+1}},
\end{equation*}
where $c_{n+1}$ depends only on the dimension and is large enough so that the right hand side of \eqref{eq:small_boundary_DM} tends to $0$ when $l\to\infty$.

Another observation is that for every $Q\in\mathcal{D}_{\mu,k}\setminus\mathcal{D}^{db}_{\mu,k}$, if $Q\in\operatorname{Ch}(\hat{Q})$ and
\begin{equation}\label{eq:l_0}
    l_0:=\max\{l\in\N:800^{l+1}\leq C_0\}=\lfloor \log_{800}C_0 \rfloor-1,
\end{equation}
if $C_0$ is large enough, from \eqref{eq:des:bolas_no_db_800_DM} one deduces
\begin{equation*}
   \mu(800 B(Q))\leq C_0^{-l_0}\mu(800B(\hat{Q})).
\end{equation*}
Given any constant $\kappa>0$, since the function $x^{-\lfloor \log_{800}x \rfloor}$ decreases faster than $x^{-\kappa}$, it is possible to take $C_0$ large enough, in terms of $\kappa$, so that
\begin{equation}\label{eq:C_0^-l_0_decrece_mas_que_A_0}
    C_0^{-l_0}<A_0^{-\kappa}.
\end{equation}
From this, the following lemma is immediate:

\begin{lema}\label{lema:acotacion_cadena_cubos_malos}
    Let $R\in\mathcal{D}_{\mu,k}$ and $Q\in \mathcal{D}_{\mu,k+l}$ be such that all cubes $Q\subsetneqq S\subsetneqq R$ satisfy $S\in\mathcal{D}_{\mu}\setminus\mathcal{D}^{db}_{\mu}$, then
    \begin{equation*}
        \mu(800B(Q))\leq A_0^{-\kappa(l-1)}\mu(800B(R)).
    \end{equation*}
\end{lema}

\begin{lema}\label{lema:cardinal_childs_en_DM}
        If $Q\in \mathcal{D}_{\mu}$, then $\#\operatorname{Ch}(Q)\leq(29C_0A_0)^{n+1}$.
    \end{lema}

    \begin{proof}
        Let $Q'\in\operatorname{Ch}(Q)$, then $Q'\subset Q$ and $29B(Q')\subset 29B(Q)$. Since $\{B(Q')\}_{Q'\in\operatorname{Ch}(Q)}$ are disjoint and taking into account \eqref{eq:control_radios_DM}, then
        \begin{equation*}
            \begin{aligned}
                \#\operatorname{Ch}(Q)c_{n+1}'\left(A_0^{-k(Q)-1}\right)^{n+1}&\leq\sum_{Q'\in\operatorname{Ch}(Q)}\mathcal{L}^{n+1}(B(Q'))=\mathcal{L}^{n+1}\left(\bigcup_{Q'\in\operatorname{Ch}(Q)}B(Q')\right)\\
                &\leq \mathcal{L}^{n+1}(29B(Q))\leq c_{n+1}'\left( 
                29C_0 A_0^{-k(Q)} \right)^{n+1},
            \end{aligned}
        \end{equation*}
        which implies the result.
    \end{proof}

\begin{lema}\label{lema:bola_asociada_DM_amb_bones_propietats}
    Let $\mu$ be a non-zero Radon measure in $\mathbb{R}^{n+1}$, $a>0$, $0<\gamma<1$,
    $C_0$ large enough depending on $a$ and the dimension, and let $B(\xi,r)$ be an $(a,P_{\mu,\gamma})$-doubling ball with $A_0^{-j-1}\leq r \leq A_0^{-j}$. If $\xi\in Q\in\mathcal{D}_{\mu,j}$, then $Q\in\mathcal{D}^{db}_{\mu}$. Moreover, $B(Q)$ is a $(C,P_{\mu,\gamma})$-doubling ball for some $C=C(n,\gamma,a,C_0)$.
\end{lema}

\begin{proof}
    By the definition of $(a,P_{\mu,\gamma})$-doubling, it is easy to check that $\mu(B(\xi,r))>0$, and for all $m\in\mathbb{N}$ it holds
    \begin{equation}\label{eq:P-doubling_implica_doubling}
        \mu(B(\xi,2^mr))\leq 2^{m(n+1)}a\mu(B(\xi,r)),
    \end{equation}
    On the other hand, by \eqref{eq:C_0^-l_0_decrece_mas_que_A_0}, we can take $C_0$ large enough so that
    \begin{equation}\label{eq:condicio_C_0_gran}
        C_0^{-l_0}<a^{-1}(4C_0A_0)^{-(n+1)},
    \end{equation}
    where $l_0$ is the constant defined in \eqref{eq:l_0}. If we suppose by contradiction that $Q\not\in\mathcal{D}_{\mu}^{db}$, then by \eqref{eq:des:bolas_no_db_800_DM}
    \begin{equation*}
        \mu(800B(Q))\leq C_0^{-l_0}\mu(800^{l_0+1}B(Q)).
    \end{equation*}
    Recall that
    \begin{equation*}
        A_0^{-j-1}\leq r \leq A_0^{-j}\quad\text{and}\quad A_0^{-j}=r(Q)\leq C_0A_0^{-j},
    \end{equation*}
    since $\xi\in Q$, it is clear that $B(\xi,r)\subset 800B(Q)$. Conversely, we claim that  $800^{l_0+1}B(Q)\subset B(\xi,2^{m_0+1}r)$. Indeed, if $y\in 800^{l_0+1}B(Q)\subset C_0 B(Q)$, then
    \begin{flalign*}
         |y-\xi|&\leq|y-x_Q|+|x_Q-\xi|\leq C_0r(Q)+28r(Q)=(C_0+28)A_0^{-j}\\
        &\leq  2C_0A_0^{-j}\leq 2C_0A_0r\leq 2^{m_0+1}r,
    \end{flalign*}
    where $x_Q$ is the center of $B(Q)$, $m_0=\lceil \log_2 C_0A_0\rceil$ and $r(Q)=A_0^{-j}$ because $Q\not\in\mathcal{D}_\mu^{db}$. So $800^{l_0+1}B(Q)\subset B(\xi,2^{m_0+1}r)$, which together with \eqref{eq:P-doubling_implica_doubling} implies that
    \begin{flalign*}
        \mu(800B(Q))&\leq C_0^{-l_0}\mu(800^{l_0+1}B(Q))\leq C_0^{-l_0}\mu(B(\xi,2^{m_0+1}r))\\
        &\leq C_0^{-l_0}2^{(m_0+1)(n+1)}a\mu(B(\xi,r))\\&
        \leq C_0^{-l_0}(4C_0A_0)^{n+1}a\mu(800B(Q))<\mu(800B(Q)).
    \end{flalign*}
    This contradiction implies that for a $C_0$ satisfying \eqref{eq:condicio_C_0_gran}, then $Q\in\mathcal{D}_\mu^{db}$.

    To show that $B(Q)$ is also $(C,P_{\mu,\gamma})$-doubling, take $i\geq0$ and $y\in 2^iB(Q)$,
    \begin{flalign*}
        |y-\xi|&\leq|y-x_Q|+|x_Q-\xi|\leq 2^ir(Q)+28r(Q)<2^{i+6}r(Q)\\
        &\leq 2^{i+6}C_0A_0^{-j}\leq 2^{i+m_0+6}r,
    \end{flalign*}
    so we can write
    \begin{flalign*}
        P_{\mu,\gamma}(B(Q))&=\sum_{i\geq0}2^{-\gamma i}\Theta^n_\mu(2^iB(Q))\leq \sum_{i\geq0}2^{-\gamma i}\dfrac{\mu(B(\xi,2^{i+6+m_0}r))}{(2^ir)^n}\\
        &=(2^{m_0+6})^n2^{\gamma(m_0+6)}\sum_{i\geq0}2^{-\gamma(i+6+m_0)}\Theta^n_\mu(B(\xi,2^{i+6+m_0}r))\\
        &\leq(2^{m_0+6})^n2^{\gamma(m_0+6)}a\Theta^n_\mu(B(\xi,r))\leq(2^{m_0+6})^n2^{\gamma(m_0+6)}a\dfrac{\mu(800B(Q))}{r^n}\\
        &\leq (2^{m_0+6})^n2^{\gamma(m_0+6)}aC_0^{n+1}A_0^n\Theta^n_\mu(B(Q)),
    \end{flalign*}
    where we used the fact that $Q\in\mathcal{D}_\mu^{db}$ in the last inequality. Observe that the $(C,P_{\mu,\gamma})$-doubling constant is bounded by
    \begin{equation*}
        C(n,\gamma,a,C_0)=(C_0A_0)^{2(n+1)+\gamma}2^{7n+7\gamma}a.
    \end{equation*}
\end{proof}

We will take 
\begin{equation*}
    \gamma=\frac{1+\gamma_0}{2},\quad a=a(\gamma,n),
\end{equation*} 
where $\gamma_0$ and $a$ are from Lemma \ref{lema:bolas_P_doblantes_medida_armonica}, so that they are absolute constants.

\section{Key theorem}
Now, we want to show the key results in order to get a contradiction since it connects the hypotheses we have assumed together with the rectifiable set that we have mentioned in the introduction. 

Our main result in this section is the following.
\begin{teo}\label{teo:main1}
    Let $\mu$ be a Radon measure in $\mathbb{R}^{n+1}$. If $B\subset\mathbb{R}^{n+1}$ is a $(C_1, P_\mu)$-doubling ball satisfying:
    \begin{enumerate}[label=\alph*)]
        \item For some $C_2>0$, there exists $G_B\subset 3B$ such that
        \begin{equation}\label{eq:acotacion-HL-Riesz}
            \sup_{0<r\leq 8 r(B)} \dfrac{\mu(B(x,r))}{r^n} +\mathcal{R}_\ast(\chi_{3B}\mu)(x)\leq C_2\Theta_\mu^n(B)
        \end{equation}
        for all $x\in 2B\cap G_B$. Moreover, $\mu(3B\setminus G_B)\leq\delta\mu(3B)$ for some $\delta>0$.
        \item There exists a ball $B_1$ centered at $B$ and there exists $\delta_0,\delta_1,\alpha>0$ such that $\Theta^n_\mu(B_1)\geq\alpha\Theta_\mu^n(B)$ with $\delta_0r(B)\leq r(B_1)\leq\delta_1 r(B)$.
        \item For some $\varepsilon>0$,
        \begin{equation*}
            \int_{2B\cap G_B}|\mathcal{R}\mu(x)-m_{\mu,2B\cap G_B}(\mathcal{R}\mu)|^2\,d\mu(x)\leq\varepsilon\Theta^n_\mu(B)^2\mu(B).
        \end{equation*}
    \end{enumerate}
    Suppose that $0<\delta_1<1/2$ small enough depending on $\alpha,C_1$ and $n$, that $\varepsilon$ is small enough, depending on $C_1,\alpha$ and $\delta_0$, and that $\delta>0$ is arbitrarily small. Then, there is some rectifiable set $\Gamma$ and some $\tau>0$ such that
    \begin{equation*}
        \mu(\Gamma\cap 2B)\geq\tau \mu(B),
    \end{equation*}
    with $\tau$ depending on the above constants. 
\end{teo}

The proof consists in reducing Theorem 3.1 to Theorem 1.5 of \cite{T}, which we recall now:

\begin{teo}\label{teo:tolsa1}
    Let $\mu$ be a Radon measure in $\mathbb{R}^{n+1}$. If $B\subset\mathbb{R}^{n+1}$ is a $(C_1, P_\mu)$-doubling ball satisfying:
    \begin{enumerate}[label=\alph*)]
        \item $\mathcal{R}(\mu|_{2B})$ is bounded in $L^2(\mu|_{2B})$ and for some $C_2>0$,  $\norm{\mathcal{R}(\mu|_{2B})}_{L^2(\mu|_{2B})\to L^2(\mu|_{2B})}\leq C_2\Theta^n_\mu(B)$ and
        \begin{equation*}
            \sup_{0<r\leq 4 r(B)} \dfrac{\mu(B(x,r))}{r^n} \leq C_2\Theta_\mu^n(B)
        \end{equation*}
        for all $x\in 2B$.
        \item There exists a ball $B_1$ centered at $B$ and there exists $\delta_0,\delta_1,\alpha>0$ such that $\Theta^n_\mu(B_1)\geq\alpha\Theta_\mu^n(B)$ with $\delta_0r(B)\leq r(B_1)\leq\delta_1 r(B)$.
        \item For some $\varepsilon>0$
        \begin{equation*}
            \int_{2B}|\mathcal{R}\mu(x)-m_{\mu,2B}(\mathcal{R}\mu)|^2\,d\mu(x)\leq\varepsilon\,\Theta^n_\mu(B)^2\mu(B).
        \end{equation*}
    \end{enumerate}
        Suppose that $0<\delta_1<1/2$ small enough depending on $\alpha,C_1$ and $n$, and that $\varepsilon$ is small enough, depending on $C_1,\alpha$ and $\delta_0$. Then there is some rectifiable set $\Gamma$ and some $\tau>0$ such that
    \begin{equation*}
        \mu(\Gamma\cap 2B)\geq\tau \mu(B),
    \end{equation*}
    with $\tau$ depending on the above constants.
\end{teo}

\begin{proof}[Proof of Theorem \ref{teo:main1}]
    Let $\tilde{\mu}=\mu|_{G_B\cup (3B)^c}$ and let us see that $\tilde\mu$ verifies the hypotheses of Theorem \ref{teo:tolsa1}. First of all, notice that since $B$ is $(C_1,P_\mu)$-doubling, 
\begin{align*}
\mu(B)& = \mu(B\cap G_B)+\mu(B\setminus G_B)\leq \tilde\mu(B)+\mu(3B\setminus G_B)\\
&\leq \tilde\mu(B)+\delta\mu(3B)\leq \tilde\mu(B)+4^{n+1}C_1\delta\mu(B).
\end{align*}   
Thus, $(1-4^{n+1}C_1\delta)\,\Theta_\mu^n(B)\leq\Theta^n_{\tilde\mu}(B)$ and so,
for $\delta<2^{-1}4^{-n-1}C_1^{-1}$, 
    \begin{equation*}
        P_{\tilde\mu}(B)\leq P_\mu(B)\leq C_1\Theta_\mu^n(B)\leq\dfrac{C_1}{1-4^{n+1}C_1\delta}\Theta_{\tilde\mu}^n(B)\leq 2C_1 \Theta_{\tilde\mu}^n(B),
    \end{equation*}
  So $B$ is $(2C_1,P_{\tilde{\mu}})$-doubling. 
  
  Let us check now that the hypothesis $b)$ also holds with the same ball $B_1$ but changing $\alpha$ by $\alpha/2$:
    \begin{flalign*}
        \alpha\Theta^n_{\tilde\mu}(B)& \leq \alpha\Theta_\mu^n(B)\leq \Theta^n_\mu(B_1)\\
        &\leq \Theta_{\tilde\mu}^n(B_1)+\dfrac{4^{n+1}C_1\delta}{r(B_1)^n}\mu(B)\leq  \Theta_{\tilde\mu}^n(B_1)+\dfrac{4^{n+1}C_1\delta}{\delta_0^n}\Theta^n_{\mu}(B).
    \end{flalign*}
    From this we conclude that if $\delta<2^{-1}\alpha4^{-n-1}C_1^{-1}\delta_0^n$, then
    \begin{equation*}
       \dfrac{\alpha}{2}\,\Theta^n_{\tilde\mu}(B)\leq \dfrac{\alpha}{2}\,\Theta^n_{\mu}(B)\leq\Theta_{\tilde\mu}^n(B_1).
    \end{equation*}

    The proof of $a)$ and $c)$ follows rather closely the ideas from \cite[Th.~3.3]{AMT}. However, we include here the full argument for completeness. 
 Define $\sigma=\mu|_{3B}$, $a,b>0$ constants to be chosen and set
    \begin{equation*}
        M_a=\{ x\in\mathbb{R}^{n+1}:\mathcal{M}_n\sigma(x)>a\Theta^n_\sigma(B) \}~\text{ and }~R_b=\{ x\in\mathbb{R}^{n+1}:\mathcal{R}_\ast\sigma(x)>b\Theta^n_\sigma(B) \}.
    \end{equation*}
    For $x\in M_a$, denote
    \begin{equation*}
        \rho_M(x)=\sup\{r>0:\sigma(B(x,r))>a\Theta^n_\sigma(B)r^n\}
    \end{equation*}
    and for $x\in R_b$
      \begin{equation*}
        \rho_R(x)=\sup\{r>0:|\mathcal{R}_r\sigma(x)|>b\Theta^n_\sigma(B)\}.
    \end{equation*}
    It is immediate to check that both $\rho_M(x)$ and $\rho_R(x)$ are finite, using that $\sigma$ is a finite and compactly supported measure.
    Now, define
    \begin{equation*}
        H_M=\bigcup_{x\in M_a}B(x,\rho_M(x))~~\text{ and }~~H_R=\bigcup_{x\in R_b}B(x,\rho_R(x)).
    \end{equation*}
    Both sets are open. We claim that $2B\cap (H_R\cup H_M)\subset 2B\setminus G_B$ if $a,b>0$ are large enough. Indeed, if $y\in 2B\cap H_M$, then there exists $x\in M_a$ such that $|x-y|<\rho_M(x)$, so
    \begin{equation}\label{eq:key_lemma_inf_bound_1}
        \sigma(B(y,2\rho_M(x)))\geq\sigma(B(x,\rho_M(x)))\geq a\Theta^n_\sigma(B)\rho_M(x)^n.
    \end{equation}
    On the other hand, if $y\in G_B~\cap ~2B$ and $2\rho_M(x)\leq8 r(B)$, then
    \begin{equation*}
        \sigma(B(y,2\rho_M(x)))\leq\mu(B(y,2\rho_M(x)))\leq C_2(2\rho_M(x))^n\Theta^n_\mu(B)=C_2(2\rho_M(x))^n\Theta^n_\sigma(B).
    \end{equation*}
    Taking $a>\max\{2^nC_2,\,4^{n+1}C_1\}$ contradicts \eqref{eq:key_lemma_inf_bound_1}. 
    
    It remains to verify that $2\rho_M(x)\leq8 r(B)$.
    If $x\in M_a$, then 
    \begin{equation*}
\mu(3B)\geq\sigma(B(x,r))>a\Theta^n_\sigma(B)r^n>C_14^{n+1}\Theta^n_\mu(B)r^n\geq\dfrac{\mu(3B)}{r(B)^n}r^n, 
    \end{equation*}
    hence $r<r(B)$, i.e., $2\rho_M(x)\leq 2r(B)\leq8 r(B)$. We conclude that $2B\cap H_M\subset 2B\setminus G_B$.

    Suppose now that $y\in 2B\cap (H_R\setminus H_M)$, so there exists $x\in R_b$ such that $y\in B(x,\rho_R(x))$. Let us show that
    \begin{equation}\label{eq:acotacions_riesz_punts_prox}
        |\mathcal{R}_{\rho_R(x)}\sigma(x)-\mathcal{R}_{\rho_R(x)}\sigma(y)|\leq C_na\Theta_\sigma^n(B).
    \end{equation}
    Notice that
    \begin{flalign*}
        &|\mathcal{R}_{\rho_R(x)}\sigma(x)-\mathcal{R}_{\rho_R(x)}\sigma(y)|\leq\\
        &\leq |\mathcal{R}_{\rho_R(x)}(\chi_{B(y,2\rho_R(x))}\sigma)(x)|+|\mathcal{R}_{\rho_R(x)}(\chi_{B(y,2\rho_R(x))}\sigma)(y)|+\\
        &+|\mathcal{R}_{\rho_R(x)}(\chi_{B(y,2\rho_R(x))^c}\sigma)(x)-\mathcal{R}_{\rho_R(x)}(\chi_{B(y,2\rho_R(x))^c}\sigma)(y)|:=I_1+I_2+I_3.
    \end{flalign*}
    Since $y\not\in H_M$, we have $\Theta^n_\sigma(B(y,r))\leq a\Theta^n_\sigma(B)$ for all $r>0$, and so we get
    \begin{equation*}
        I_1\leq\int_{|x-z|>\rho_R(x)}\dfrac{\chi_{B(y,2\rho_R(x))}(z)}{|x-z|^n}\,d\sigma(z)\leq2^n\dfrac{\sigma(B(y,2\rho_R(x)))}{2^n\rho_R(x)^n}\leq 2^na\Theta^n_\sigma(B),
    \end{equation*}
    and
    \begin{equation*}
        I_2\leq\int_{|y-z|>\rho_R(x)}\dfrac{\chi_{B(y,2\rho_R(x))}(z)}{|y-z|^n}\,d\sigma(z)\leq2^n\dfrac{\sigma(B(y,2\rho_R(x)))}{2^n\rho_R(x)^n}\leq 2^na\Theta^n_\sigma(B).
    \end{equation*}
    For the last one, using the H\"older continuity estimate for Calder\'on-–Zygmund kernels, we obtain
    \begin{flalign*}
        I_3&=\left| \int_{|z-y|\geq2\rho_R(x)}\dfrac{x-z}{|x-z|^{n+1}}\,d\sigma(z)-\int_{|z-y|\geq2\rho_R(x)}\dfrac{y-z}{|y-z|^{n+1}}\,d\sigma(z) \right|\\
        &\leq  \tilde{C}_n\int_{|z-y|\geq2\rho_R(x)}\dfrac{|x-y|}{|y-z|^{n+1}}\,d\sigma(z)\leq \tilde{C}_n\sum_{j\geq1}\dfrac{\rho_R(x)}{(2^j\rho_R(x))^{n+1}}\sigma(B(y,2^{j+1}\rho_R(x)))\\
         &\leq \tilde{C}_n\sum_{j\geq1}\dfrac{2^n}{2^j}a\Theta^n_\sigma(B)=\tilde{C}_n2^na\Theta^n_\sigma(B),
    \end{flalign*}
    so from \eqref{eq:acotacions_riesz_punts_prox}, we get
    \begin{equation*}
        |\mathcal{R}_{\rho_R(x)}\sigma(y)|\geq|\mathcal{R}_{\rho_R(x)}\sigma(x)|-C_n a\Theta^n_\sigma(B)\geq b\Theta^n_\sigma(B)-C_na\Theta^n_\sigma(B).
    \end{equation*}
    If $b-C_na>C_2$, we conclude that $|\mathcal{R}_{\rho_R(x)}\sigma(y)|>C_2\Theta_\sigma^n(B)$. This implies that $2B\cap H_R\setminus H_M\subset2B\setminus G_B$.

    Taking this $a,b>0$ large enough so that $2B\cap(H_R\cup H_M)\subset2B\setminus G_B$, we consider the open set $H=H_R\cup H_M$ and we define  the $1$-Lipschitz function
    \begin{equation*}
        \Phi(x)=\operatorname{dist}(x,H^c)\geq\max\{\rho_R(x)\,,\rho_M(x)\},
    \end{equation*}
    and the associated ``suppressed kernel"
    \begin{equation*}
        K_\Phi(x,y)=\dfrac{x-y}{(|x-y|^2+\Phi(x)\Phi(y))^{\frac{n+1}{2}}}.
    \end{equation*}
    Consider the operator
    \begin{equation*}
        \mathcal{R}_{\Phi,\sigma}f(x):=\int K_\Phi(x,y)f(y)\,d\sigma(y),
    \end{equation*}
    and for every $\varepsilon>0$, its truncated version
     \begin{equation*}
        \mathcal{R}_{\Phi,\varepsilon,\sigma}f(x):=\int_{|x-y|>\varepsilon} K_\Phi(x,y)f(y)\,d\sigma(y).
    \end{equation*}
    Also set
    \begin{equation*}
        \mathcal{R}_{\Phi,\ast,\sigma}f(x)=\sup_{\varepsilon>0}\{|\mathcal{R}_{\Phi,\varepsilon,\sigma}f(x)|\}.
    \end{equation*}
    We say that $\mathcal{R}_{\Phi,\sigma}$ is bounded in $L^2(\sigma)$ if the operators $\mathcal{R}_{\Phi,\varepsilon,\sigma}$ are bounded in $L^2(\sigma)$ uniformly in $\varepsilon>0$.
    We now prove that
    \begin{equation}\label{eq:acotacion_riesz_supressed}
        \mathcal{R}_{\Phi,\ast,\sigma}1(x)\leq C(a,b)\Theta^n_\sigma(B),
    \end{equation}
    for all $x\in\mathbb{R}^{n+1}$. To do so, we need the following lemma, the proof of which can be found in \cite[L.~5.5]{CZ-T}
    \begin{lema}\label{lema:acotacion_riesz_supressed}
        Let $x\in\mathbb{R}^{n+1}$ and $r_0\geq0$ be such that $\sigma(B(x,r))\leq A_1r^n$ for $r\geq r_0$ and $|\mathcal{R}_\varepsilon\sigma(x)|\leq A_2$ for $\varepsilon\geq r_0$. If $\Phi(x)\geq r_0$, then there exists $C>0$, so that $|\mathcal{R}_{\Phi,\varepsilon,\sigma}1(x)|\leq CA_1+A_2$ for all $\varepsilon>0$.
    \end{lema}
    Now observe that taking $A_1=a\Theta^n_\sigma(B)$, $A_2=b\Theta_\sigma^n(B)$ and $r_0=\max\{ \rho_R(x),\rho_M(x) \}$ in Lemma \ref{lema:acotacion_riesz_supressed}, we obtain \eqref{eq:acotacion_riesz_supressed}. We further apply the $Tb$ theorem for suppressed operators proved in \cite{NTV} (see also \cite[Cor.~5.33]{CZ-T}) and it follows that $\mathcal{R}_{\Phi,\sigma}:L^2(\sigma)\rightarrow L^2(\sigma)$ is bounded with norm
    \begin{equation*}
        \norm{\mathcal{R}_{\Phi,\sigma}}_{L^2(\sigma)\to L^2(\sigma)}\lesssim\Theta_\sigma^n(B)\leq2\Theta_{\tilde{\mu}}^n(B).
    \end{equation*}
    Since $\Phi(x)=0$ for all $x\in 2B\cap G_B\subset H^c$, we have $\mathcal{R}_{\tilde{\mu}|_{2B}}:L^2(\tilde{\mu}|_{2B})\rightarrow L^2(\tilde{\mu}|_{2B})$
    with $\norm{\mathcal{R}_{\tilde{\mu}|_{2B}}}_{L^2(\tilde{\mu}|_{2B})\to L^2(\tilde{\mu}|_{2B})}\lesssim\Theta_{\tilde{\mu}}^n(B)$.

    For the density estimate in the hypothesis $a)$, take $x\in 2B$ and $0<r\leq 4r(B)$. If $x\in\overline{G_B}$, then take $y_r\in B(x,r)\cap G_B$ and observe that
    \begin{equation*}
        \Theta_{\tilde{\mu}}^n(B(x,r))\leq2^n\Theta^n_\mu(B(y_r,2r))\leq 2^nC_2\Theta_\mu^n(B)\leq2^{n+1}C_2\Theta^n_{\tilde{\mu}}(B). 
    \end{equation*}
    In case that $x\not\in \overline{G_B}$, since $d(x,\partial3B)\geq r(B)$, either $B(x,r)\cap G_B\neq\emptyset$ for $0<r<r(B)$, so we argue as before, or $\tilde{\mu}(B(x,r))=0$ for all $0<r<r(B)$, and then in this case
    \begin{equation*}
    \begin{aligned}
        \sup_{0<r\leq4r(B)}\{\Theta_{\tilde{\mu}}^n(B(x,r))\}&=\sup_{r(B)\leq r\leq4r(B)}\{\Theta_{\tilde{\mu}}^n(B(x,r))\}
        \\
        &\leq\dfrac{\mu(6B)}{r(B)^n}\leq8^{n+1}C_1\Theta^n_\mu(B)\leq 8^{n+1}2C_1\Theta^n_{\tilde{\mu}}(B).
    \end{aligned}
    \end{equation*}

    Condition $c)$ remains, taking into account that
    \begin{flalign*}
        \int_{2B}|\mathcal{R}\tilde{\mu}(x)-m_{\tilde{\mu},2B}(\mathcal{R}\tilde{\mu})|^2\,d\tilde{\mu}(x)&\lesssim\int_{2B\cap G_B}|\mathcal{R}{\mu}(x)-m_{{\mu},2B\cap G_B}(\mathcal{R}{\mu})|^2\,d{\mu}(x)\\
        &\quad+\int_{2B\cap G_B}|\mathcal{R}(\mu-\tilde{\mu})(x)|^2\,d\mu(x)=J_1+J_2.
    \end{flalign*}
    By hypothesis, we already know $J_1\leq\varepsilon\Theta^n_\mu(B)^2\mu(B)\leq 2^3\varepsilon\Theta^n_{\tilde{\mu}}(B)^2\tilde\mu(B)$.

    For $J_2$, notice that $\mu-\tilde\mu=\mu|_{3B\setminus G_B}$ and, furthermore, recall that $\Phi$ vanishes on $2B\cap G_B\subset H^c$, so $\mathcal{R}(\mu-\tilde\mu)(x)=\mathcal{R}_\Phi(\mu-\tilde\mu)(x)$ for all $x\in 2B\cap G_B$. Using now the boundedness of $\mathcal{R}_{\Phi,\sigma}$ from $L^1(\sigma)$ to $L^{1,\infty}(\sigma)$ (see \cite[L.~5.27]{CZ-T}) and duality, we can conclude that $\mathcal{R}_{\Phi,\sigma}$ is bounded in $L^4(\sigma)$, so we have that
    \begin{flalign*}
        \int_{2B\cap G_B}|\mathcal{R}(\mu-\tilde\mu)|^2\,d\mu&\leq\mu(G_B)^{1/2}\|\mathcal{R}_{\Phi}(\mu|_{3B\setminus G_B})\|^2_{L^4(\mu|_{2B\cap G_B})}\\
        &\lesssim \Theta^n_\sigma(B)^2\mu(G_B)^{1/2}\mu(3B\setminus G_B)^{1/2}\lesssim\delta^{1/2}\Theta^n_{\tilde\mu}(B)^2\tilde\mu(B).
    \end{flalign*}
    Gathering the estimates for $J_1$ and $J_2$ we obtain
    \begin{equation*}
        \int_{2B}|\mathcal{R}\tilde{\mu}(x)-m_{\tilde{\mu},2B}(\mathcal{R}\tilde{\mu})|^2\,d\tilde{\mu}(x)\lesssim(\varepsilon+\delta^{1/2})\Theta^n_{\tilde\mu}(B)^2\tilde\mu(B).
    \end{equation*}
\end{proof}

\section{Reduction to Wiener regular open sets and $\mathbb{R}^{n+1}$ for $n\geq2$}\label{sec4}

In this section, we will show that the planar case of Theorem \ref{teo:main_result} is trivial and moreover, that if this is true for Wiener regular open sets $\Omega$ with compact boundary, then it also holds for general open sets with compact boundary. 

In the case $n=1$, we recall Wolff's result in \cite{W2} which states that given an open set $\Omega\subset\mathbb{R}^2$, and $p\in\Omega$, there exists a set $F\subset\partial\Omega$ and a family $\{F_i\}_{i\in\mathbb{N}}$ such that $\omega^p_\Omega(F)=1$,  $F=\bigcup_{i\in\mathbb{N}}F_i$ and $\mathcal{H}^1(F_i)<\infty$. If
\begin{equation*}
    Z=\left\{x\in\mathbb{R}^{2}~:~\lim_{r\to0}\dfrac{\omega_\Omega^p(B(x,r))}{r}=0\right\},
\end{equation*}
then  $\omega^p_\Omega(Z)=0$ by the Besicovitch density theorem. Indeed, 
it is clear that
\begin{equation*}
    \Theta^{\ast,1}_{\omega^p_\Omega}(x)=\limsup_{r\to0}\dfrac{\omega^p_\Omega(B(x,r))}{r}=0\quad\text{ for all }~x\in Z,
\end{equation*}
and so  by \cite[Th.~6.9]{MAT}, for all $\varepsilon>0$ and all $i\in\mathbb{N}$
\begin{equation*}
    \omega^p_\Omega(Z\cap F_i)\leq 2\varepsilon \mathcal{H}^1(Z\cap F_i),
\end{equation*}
so we conclude that $\omega^p_\Omega(Z)=0$. This proves Theorem \ref{teo:main_result} in the case $n=1$.

Suppose that $n\geq2$, $\Omega\subset\mathbb{R}^{n+1}$ is an open set with compact boundary, $p\in\Omega$. Recall that
\begin{equation*}
    Z=\left\{x\in\mathbb{R}^{n+1}~:~\lim_{r\to0}\dfrac{\omega_\Omega^p(B(x,r))}{r^n}=0\right\}.
    \end{equation*}
Moreover, define
\begin{equation*}
    \mathcal{O}=\left\{x\in Z~:~\limsup_{R\to0}\sup_{0<r<R}\dfrac{\Theta^n_{\omega_\Omega^p}(x,r)}{\Theta^n_{\omega_\Omega^p}(x,R)}<\infty\right\}.
\end{equation*}
Assume Theorem \ref{teo:main_result} to be valid under the additional assumption of Wiener regularity and suppose that $\omega^p_\Omega(\mathcal{O})>0$. For every $\varepsilon>0$ we can consider the modified domain $\Omega_\varepsilon$ defined in Lemma \ref{lema:modificacion_dominios_no_WR}. Then, by Lemma \ref{lema:aproximacion_medida_con_dominios_WR_modificados} and (vi) in Lemma \ref{lema:modificacion_dominios_no_WR}, we can take $\varepsilon>0$ sufficiently small such that 
\begin{equation}\label{eq:reduction_wr}
    \omega^p_{\Omega_\varepsilon}\left(\partial\Omega\setminus\bigcup_{i\in I_\varepsilon}2\overline{B_i}\right)\geq \omega^p_{\Omega_\varepsilon}\left(\mathcal{O}\setminus\bigcup_{i\in I_\varepsilon}2\overline{B_i}\right)\geq \frac12\, \omega^p_{\Omega_\varepsilon}(\mathcal{O})>0.
\end{equation}
By Lemma \ref{lema:Markov_property_harmonic_measure}, the set $A=\partial\Omega\setminus\bigcup_{i\in I_\varepsilon}2\overline{B_i}$ satisfies $\omega_{\Omega_\varepsilon}^p|_A\ll\omega^p_\Omega|_A$. It is clear that $\omega^p_{\Omega_\varepsilon}$-a.e. $x\in A$ is a Lebesgue point of $A$ for $\omega^p_\Omega$ and $\omega^p_{\Omega_\varepsilon}$ and moreover, a Lebesgue point of $h_\varepsilon:=\frac{d\omega^p_{\Omega_\varepsilon}}{d\omega^p_\Omega}$ for $\omega^p_\Omega|_A$. Such point satisfies
\begin{flalign*}
    &\lim_{r\to0}\dfrac{\omega^p_{\Omega}(B(x,r)\cap A)}{\omega^p_{\Omega}(B(x,r))}= \lim_{r\to0}\dfrac{\omega^p_{\Omega_\varepsilon}(B(x,r)\cap A)}{\omega^p_{\Omega_\varepsilon}(B(x,r))} = 1,
\end{flalign*}
so in particular,
\begin{align*}
    \lim_{r\to0}\dfrac{\Theta^n_{\omega^p_\Omega}(x,r)}{\Theta^n_{\omega^p_{\Omega_\varepsilon}}(x,r)}   & = \lim_{r\to0} \dfrac{\omega^p_\Omega(B(x,r))}{\omega^p_{\Omega_\varepsilon}(B(x,r))}\\
    &=
    \lim_{r\to0} \dfrac{\omega^p_\Omega(B(x,r)\cap A)}{\omega^p_{\Omega_\varepsilon}(B(x,r)\cap A)}
=\frac1{h_\varepsilon(x)}\in(0,+\infty) \,\quad\omega^p_{\Omega_\varepsilon}\text{-a.e.}~~x\in A.
\end{align*}
Since the ratio of the two density ratios converges to a finite positive constant, the vanishing of the density and the above almost-monotonicity property are preserved when passing from $\omega^p_\Omega$ to $\omega^p_{\Omega_\varepsilon}$.

 Since $\mathcal{O}\setminus\bigcup_{i\in I_\varepsilon}2\overline{B_i}\subset A$, we get
\begin{equation*}
    \limsup_{R\to0}\sup_{0<r<R}\frac{\Theta^n_{\omega_{\Omega_\varepsilon}^p}(x,r)}{\Theta^n_{\omega_{\Omega_\varepsilon}^p}(x,R)}<\infty\quad\omega^p_{\Omega_\varepsilon}\text{-a.e.}~~x\in \mathcal{O}\setminus\bigcup_{i\in I_\varepsilon}2\overline{B_i}.
\end{equation*}
By assumption, since $\Omega_\varepsilon$ is Wiener regular, we conclude that
\begin{equation*}
    \omega^p_{\Omega_\varepsilon}\left(\mathcal{O}\setminus\bigcup_{i\in I_\varepsilon}2\overline{B_i}\right)=0.
\end{equation*}
This contradicts  \eqref{eq:reduction_wr}.

\section{Proof of Theorem \ref{teo:main_result}}

We may assume that $n\geq 2$ and that $\Omega$ is a Wiener regular set with compact boundary. Moreover, for a fixed $p\in\Omega$, we have $\omega^p_\Omega=\omega^p_{\mathbb{R}^{n+1}\setminus\partial\Omega}$, so replacing $\Omega$ by $\mathbb{R}^{n+1}\setminus \partial \Omega$ we may suppose that $\Omega=\mathbb{R}^{n+1}\setminus\partial\Omega$ without loss of generality. In this situation, we can ensure that the component of $\Omega$ that contains $p$ is Wiener regular.

Again, we consider the sets
\begin{equation*}
    Z=\left\{x\in\mathbb{R}^{n+1}:\lim_{r\to0}\dfrac{\omega_\Omega^p(B(x,r))}{r^n}=0\right\}
\end{equation*}
and
\begin{equation*}
    \mathcal{O}=\left\{x\in Z:\limsup_{R\to0}\sup_{0<r<R}\dfrac{\Theta^n_{\omega_\Omega^p}(x,r)}{\Theta^n_{\omega_\Omega^p}(x,R)}<\infty\right\}.
\end{equation*}
We argue by contradiction. Assume that $\omega_\Omega^p(\mathcal{O})>0$. Then, there exist $m\in\mathbb{N}$ and $\beta>0$ such that the set
\begin{equation*}
    E_{m,\beta}:=\left\{x\in Z:\sup_{0<r<R}\dfrac{\Theta^n_{\omega^p_\Omega}(x,r)}{\Theta^n_{\omega^p_\Omega}(x,R)}\leq m \quad\text{for all } R\leq\beta\right\}
\end{equation*}
satisfies $\omega^p_\Omega(E_{m,\beta})>0$.

To reach a contradiction, we will apply Theorem \ref{teo:main1} to a suitable small ball $B$ centered at a density point of $E_{m,\beta}$. That theorem yields the existence of a rectifiable set $\Gamma$ such that $\mu(\Gamma\cap 2B)\geq\tau \mu(B)$ for some $\tau>0$, which will contradict the fact that points in $\mathcal{O}$ have vanishing density for $\omega^p_\Omega$. In order to apply Theorem \ref{teo:main1}, we will show that hypotheses a) and c) hold by following ideas similar to those in \cite{AMT, AMTV}. The main difficulty lies in proving the existence of a ball $B_1$ fulfilling the assumptions of b), that is, $B_1$ has to be sufficiently small relative to $B$ and must have a density ratio bounded from below.\footnote{We remark that if $\Omega$ satisfied the CDC, then it would be much easier to check all the required assumptions.}

We consider the David--Mattila lattice $\mathcal{D}_{\omega^p_\Omega}$ introduced in Section \ref{sec2.3}. As mentioned in that section, we take some large constants $C_0$ and $A_0=C_0^{c_{n+1}}$. We choose
\begin{equation*}
\kappa = 8(n+1),
\end{equation*}
and we take $C_0$ so large that  
\eqref{eq:C_0^-l_0_decrece_mas_que_A_0} and Lemma \ref{lema:acotacion_cadena_cubos_malos} hold with 
 $\kappa=8(n+1)$. We will not keep track of the constants $A_0$ and $C_0$ explicitly if it is not necessary.

Now, let $\xi_0\in \partial\Omega$ be a Lebesgue density point of $E_{m,\beta}$ with respect to $\omega^p_\Omega$. By Lemma \ref{lema:bolas_P_doblantes_medida_armonica}, we can assume that there exists a sequence of $(a, P_{\gamma,\omega^p_\Omega})$-doubling balls $B(\xi_0,r_k)$ with radii $r_k\to 0$, for some constants $0<\gamma<1$ and $a>0$. 
On the other hand, for any $\delta\in(0,1/2)$ arbitrarily small, we can find $k_0=k_0(\delta,\beta,p)\in\mathbb{N}$ such that
\begin{enumerate}[label=\roman*)]
    \item $p\notin B(\xi_0,570C_0^2A_0r_{k_0})$;\label{enum:pol_p}
    \item $570C_0^2A_0r_{k_0}<\beta$;\label{enum:beta}
    \item $\omega_\Omega^p(B(\xi_0,r)\setminus E_{m,\beta})<\delta 2^{-11(n+1)}a^{-1}(C_0A_0)^{-(n+1)}\omega_\Omega^p(B(\xi_0,r))$ whenever $r\leq 570C_0^2A_0r_{k_0}$.\label{enum:delta}
\end{enumerate}
Let $j_0\in\mathbb{N}$ be such that $A_0^{-j_0-1}\leq r_{k_0}<A_0^{-j_0}$ and take $R\in \mathcal{D}_{\omega^p_\Omega,j_0}$ as in Lemma \ref{lema:bola_asociada_DM_amb_bones_propietats}, so that $\xi_0\in R\in\mathcal{D}^{db}_{\omega^p_\Omega}$ and $B(R)$ is a $(C,P_{\omega^p_\Omega,\gamma})$-doubling ball for some $C=C(n,\gamma,a,C_0)$. It is easy to check that 
\begin{equation*}
    B(\xi_0,r_{k_0})\subset 800B(R)\subset B(\xi_0,828C_0A_0r_{k_0}).
\end{equation*}
Therefore, using these inclusions, the fact that $B(\xi_0,r_{k_0})$ is $(a,P_{\omega^p_\Omega,\gamma})$-doubling, \eqref{eq:P-doubling_implica_doubling}, and $828C_0A_0\leq 570C_0^2A_0$ in \ref{enum:delta}, we obtain
\begin{flalign}\label{eq:800B(R)_tiene_masa_en_Emb}
    \omega^p_\Omega(800B(R)\setminus E_{m,\beta})&\leq \omega^p_\Omega(B(\xi_0,828C_0A_0r_{k_0})\setminus E_{m,\beta})\notag\\
    &<\delta 2^{-11(n+1)}a^{-1}(C_0A_0)^{-(n+1)}\omega^p_\Omega(B(\xi_0,828C_0A_0r_{k_0}))\notag\\
    &\leq\delta 2^{-11(n+1)}a^{-1}(C_0A_0)^{-(n+1)}\omega^p_\Omega(B(\xi_0,2^{10+\lceil\log_2(C_0A_0)\rceil}r_{k_0}))\notag\\
    &\leq \delta\omega^p_\Omega(B(\xi_0,r_{k_0}))\leq\delta \omega^p_\Omega(800B(R));
\end{flalign}
thus, most of the harmonic measure of $800B(R)$ is concentrated on $E_{m,\beta}$ in terms of $\delta$. 

For any cube $Q\in\mathcal{D}_{\omega^p_\Omega}$, let $B_Q:=29B(Q)$. From the scaling properties of the David--Mattila lattice, one can check that if $Q\subset P$, then $B_Q\subset B_P$. Now, we define the family of neighbor cubes of $R$ and the enlarged region $R'$ as
\begin{equation*}
    \mathcal{N}(R)=\{Q\in\mathcal{D}_{\omega^p_\Omega,j_0}:Q\cap 3B_R\neq\emptyset\},\quad R'=\bigcup_{\tilde{R}\in\mathcal{N}(R)}\tilde{R},
\end{equation*}
and
\begin{equation*}
    \mathcal{D}_{\omega^p_\Omega}(R')=\bigcup_{\tilde{R}\in\mathcal{N}(R)}\mathcal{D}_{\omega^p_\Omega}(\tilde{R}).
\end{equation*}
For all $\tilde{R}\in\mathcal{N}(R)$, let
\begin{equation*}
    \mathcal{B}_{\tilde{R}}=\{Q\in\mathcal{D}_{\omega^p_\Omega}(\tilde{R}):\omega^p_\Omega(Q\setminus E_{m,\beta})\geq\tfrac{1}{2C_0}\omega^p_\Omega(Q)\}.
\end{equation*}
We consider $B_R$ as our main ball for Theorem \ref{teo:main1} and set
\begin{equation*}
    G_B=G_{B_R}=(3B_R\cap E_{m,\beta})\setminus\bigcup_{\tilde{R}\in\mathcal{N}(R)}\bigcup_{Q\in\mathcal{B}_{\tilde{R}}}Q.
\end{equation*}

\begin{lema}\label{lema:G_B_mucha_masa_dentro_de_3B}
    There exists $\delta'=\delta'(n,C_0,\delta)\in(0,1)$ such that
    \begin{equation*}
        \omega^p_\Omega(3B_R\setminus G_B)<\delta'\omega^p_\Omega(3B_R),
    \end{equation*}
    assuming $\delta$ is small enough.
\end{lema}

\begin{proof}
    By definition, $3B_R\setminus G_B\subset (3B_R\setminus E_{m,\beta})\cup\bigcup_{\tilde{R}\in\mathcal{N}(R)}\bigcup_{Q\in\mathcal{B}_{\tilde{R}}}Q$. By \eqref{eq:800B(R)_tiene_masa_en_Emb},
    \begin{equation*}
        \omega^p_\Omega(800B(R)\setminus E_{m,\beta})<\delta\omega^p_\Omega(800B(R)).
    \end{equation*}
    Combining this with the fact that $R\in\mathcal{D}^{db}_{\omega^p_\Omega}$, it follows that
    \begin{equation*}
        \omega^p_\Omega(3B_R\setminus E_{m,\beta})\leq \delta \,\omega^p_\Omega(800B(R))\leq \delta C_0\omega^p_\Omega(3B_R).
    \end{equation*}
    On the other hand, since all the cubes in $\mathcal{N}(R)$ are at the same scale $j_0$, by the geometric properties of the David--Mattila lattice, one checks that 
    \begin{equation*}
        R'\subset 57C_0B(R)\subset B(\xi_0,85C_0^2A_0 r_{k_0}),
    \end{equation*}
    and moreover, arguing as above with $800B(R)\setminus E_{m,\beta}$, we deduce
    \begin{flalign}\label{eq:R'_tiene_casi_toda_la_masa_en_E_mb}
        \omega^p_\Omega(R'\setminus E_{m,\beta})&\leq\omega^p_\Omega(57C_0B(R)\setminus E_{m,\beta})\notag\\
        &\leq\omega^p_\Omega(B(\xi_0,85C_0^2A_0 r_{k_0})\setminus E_{m,\beta})\leq C_0^{n+1}\delta\omega^p_\Omega(3B_R).
    \end{flalign}
    For all $\tilde{R}\in\mathcal{N}(R)$ and all $Q\in\mathcal{B}_{\tilde{R}}$, let $\breve{Q}$ denote the maximal cube in $\mathcal{B}_{\tilde{R}}$ that contains $Q$. Then we get
    \begin{flalign*} 
        \omega^p_\Omega\left(\bigcup_{\tilde{R}\in\mathcal{N}(R)}\bigcup_{Q\in\mathcal{B}_{\tilde{R}}}Q\right)&=\omega^p_\Omega\left(\bigcup_{\tilde{R}\in\mathcal{N}(R)}\bigcup_{\breve{Q}\in\mathcal{B}_{\tilde{R}}}\breve{Q}\right)\\
        &=\sum_{\tilde{R}\in\mathcal{N}(R)}\sum_{\breve{Q}\in\mathcal{B}_{\tilde{R}}}\omega^p_\Omega(\breve{Q})
        \leq 2C_0\sum_{\tilde{R}\in\mathcal{N}(R)}\sum_{\breve{Q}\in\mathcal{B}_{\tilde{R}}}\omega^p_\Omega(\breve{Q}\setminus E_{m,\beta})\\
        &=2C_0\omega^p_\Omega\left(\bigcup_{\tilde{R}\in\mathcal{N}(R)}\bigcup_{\breve{Q}\in\mathcal{B}_{\tilde{R}}}\breve{Q}\setminus E_{m,\beta}\right)\leq 2C_0\omega^p_\Omega\left(R'\setminus E_{m,\beta}\right)\\
        &\overset{\eqref{eq:R'_tiene_casi_toda_la_masa_en_E_mb}}{\leq} 2C_0^{n+2}\delta\omega^p_\Omega(3B_R).
    \end{flalign*}
    Therefore, taking $\delta'=(2C_0^{n+2}+C_0)\delta$ concludes the proof.
\end{proof}

From now on, we assume that $\delta$ has been chosen sufficiently small so that $0<\delta'\ll 1$. Several of the arguments that follow were stated under the CDC assumption in \cite{AMT, HM}. The following lemma provides the key substitute for the capacity density condition in all these arguments.

\begin{lema}\label{lema:densidad_acotada_cdc_local}
    Let $\tilde{R}\in\mathcal{N}(R)$ and $Q\in \mathcal{D}_{\omega^p_\Omega}(\tilde{R})$ be such that $Q\in\mathcal{D}^{db}_{\omega^p_\Omega}\setminus \mathcal{B}_{\tilde{R}}$. Then $\Capp(\overline{B(Q)}\setminus\Omega)\gtrsim_{n,m,C_0} r(Q)^{n-1}$.
\end{lema}

\begin{proof}
    First, note that if $Q\in \mathcal{D}_{\omega^p_\Omega}(\tilde{R})$, we deduce that $r(Q)\leq r(\tilde{R})\leq C_0 A_0^{-j_0}\leq\beta/800$ by \ref{enum:beta}. Now, if $Q\notin\mathcal{B}_{\tilde{R}}$ and $Q\in\mathcal{D}^{db}_{\omega^p_\Omega}$, then
    \begin{equation}\label{eq:Si_Q_tiene_medida_en_E_B(Q)_tambien}
        \omega^p_\Omega(B(Q)\setminus E_{m,\beta})\leq\omega^p_\Omega(Q\setminus E_{m,\beta})<\frac{1}{2C_0}\omega^p_\Omega(Q)\leq \frac{1}{2}\omega^p_\Omega(B(Q)).
    \end{equation}
    Consider the measure $\mu=\omega_\Omega^p|_{B(Q)\cap E_{m,\beta}}$.
  
    For 
    $x\in\mathbb{R}^{n+1}$ and  $0<r<\frac{r(Q)}{2}$ such that $B(x,r)\cap \operatorname{supp}\mu \neq \emptyset$, choose $x_0\in B(x,r)\cap B(Q)\cap E_{m,\beta}$. By the definition of $E_{m,\beta}$, we have    
       \begin{equation*}
        \Theta^n_{\mu}(B(x,r))\leq 2^n\Theta^n_{\mu}(x_0,2r)\leq 2^n m\,\Theta^n_{\omega^p_\Omega}(B(x_0,r(Q)))\leq 2^n m C_0\Theta^n_{\omega^p_\Omega}(B(Q)),
    \end{equation*}
    since $2r\leq \beta$.
    If instead $r\geq\frac{r(Q)}{2}$, then \begin{equation*}
    \Theta^n_{\mu}(B(x,r))\leq 2^n\Theta^n_{\omega^p_\Omega}(B(Q))\leq 2^n m C_0\Theta^n_{\omega^p_\Omega}(B(Q)).
    \end{equation*}
    Therefore, the measure $\sigma=(2^n m C_0\Theta^n_{\omega^p_\Omega}(B(Q)))^{-1}\mu$ satisfies $\mathcal{M}_n\sigma(x)\leq 1$ for all $x\in \mathbb{R}^{n+1}$. Moreover, $\operatorname{supp}\sigma\subset \overline{B(Q)\cap E_{m,\beta}}\subset \overline{B(Q)}\setminus\Omega$. Hence, Frostman's lemma (see \cite[Chapter 4]{HM}) and \eqref{eq:Si_Q_tiene_medida_en_E_B(Q)_tambien} yield
    \begin{flalign*}
        \mathcal{H}^n_\infty(\overline{B(Q)}\setminus\Omega)&\gtrsim_n \sigma(\overline{B(Q)}\setminus\Omega)=(2^n m C_0\Theta^n_{\omega^p_\Omega}(B(Q)))^{-1}\omega^p_\Omega (B(Q)\cap E_{m,\beta})\\
        &\geq r(Q)^n (2^n m C_0)^{-1}\dfrac{\frac{1}{2}\omega^p_\Omega(B(Q))}{\omega^p_\Omega(B(Q))}\gtrsim_{n,m,C_0} r(Q)^n.
    \end{flalign*}
    Combining this estimate with Lemma \ref{lema:relacion_contenido_Hausodrf_capacidad} gives the desired result.
\end{proof}

We are now ready to verify that $\omega^p_\Omega$, $B_R$, and $G_B$ satisfy the assumptions of Theorem \ref{teo:main1}.

\subsection{Hypothesis a)}
Let $x\in G_B\cap 2B_R\subset E_{m,\beta}$ and $0<r\leq 8r(B_R)=232 r(R)<\beta$ by \ref{enum:beta}. Thus,
\begin{equation*}
    \Theta^n_{\omega^p_\Omega}(x,r)\leq m \Theta^n_{\omega^p_\Omega}(x,8r(B_R))\leq \frac{m}{8^n}\dfrac{\omega^p_\Omega(10B_R)}{r(B_R)^n}\leq C_0 m\Theta^n_{\omega^p_\Omega}(B_R).
\end{equation*}
Next, we estimate $\mathcal{R}_\ast(\chi_{3B_R}\omega^p_\Omega)(x)$ for all $x\in 2B_R\cap G_B$. Observe that for every $r>5r(B_R)$, we have $|\mathcal{R}_r(\chi_{3B_R}\omega^p_\Omega)(x)|=0$. Moreover, if $\frac{r(B_R)}{A_0}\leq r\leq 5r(B_R)$, then
\begin{equation*}
    |\mathcal{R}_r(\chi_{3B_R}\omega^p_\Omega)(x)|\leq A_0^n\dfrac{\omega^p_\Omega(3B_R)}{r(B_R)^n}\leq A_0^n C_0\Theta^n_{\omega^p_\Omega}(B_R).
\end{equation*}
Now assume that $0<r<\frac{r(B_R)}{A_0}$ and split the term as follows:
\begin{flalign*}
    |\mathcal{R}_r(\chi_{3B_R}\omega^p_\Omega)(x)| &= |\mathcal{R}_r\omega^p_\Omega(x)-\mathcal{R}_r(\chi_{(3B_R)^c}\omega^p_\Omega)(x)| \\
    &\leq |\mathcal{R}_r\omega^p_\Omega(x)-\mathcal{R}_{A_0^{-1}r(B_R)}\omega^p_\Omega(x)| + |\mathcal{R}_{A_0^{-1}r(B_R)}\omega^p_\Omega(x)-\mathcal{R}_r(\chi_{(3B_R)^c}\omega^p_\Omega)(x)| \\
    &=: I + II.
\end{flalign*}
For the second term, we easily obtain
\begin{equation*}
    II\leq\int_{3B_R,\,|x-y|>A_0^{-1}r(B_R)}\dfrac{d\omega^p_\Omega(y)}{|x-y|^n}\leq A_0^n\dfrac{\omega^p_\Omega(3B_R)}{r(B_R)^n}\leq A_0^n C_0\Theta^n_{\omega^p_\Omega}(B_R).
\end{equation*}
Estimating the first term $I$ is more involved. Take $i_0>j_0$ such that $A_0^{-i_0}\leq r < A_0^{-i_0+1}$. We can find a cube $Q_1\in\mathcal{D}_{\omega^p_\Omega,i_0-1}$ with $x\in Q_1$, and a chain of cubes $Q_1\subsetneqq Q_2\subsetneqq\cdots\subsetneqq Q_N=\tilde{R}$ for some $\tilde{R}\in\mathcal{N}(R)$, with $N=i_0-j_0$. One can check that
\begin{equation}\label{eq:inclusiones_bolas_riesz_DM}
    B(x,A_0^{l}r)\subset B_{Q_{l+1}}\quad\text{for all } 0\leq l\leq N-1.
\end{equation}
Since $x\in G_B$, $Q_l\notin\mathcal{B}_{\tilde{R}}$ for all $1\leq l\leq N$. To apply Lemma \ref{lema:densidad_acotada_cdc_local}, we rely on doubling cubes and distinguish three cases depending on the first doubling cube along this chain.

\medskip
\noindent\textbf{Case 1:} $Q_1\in\mathcal{D}_{\omega^p_\Omega}^{db}$. Let $\varphi:\mathbb{R}^{n+1}\to[0,1]$ be a smooth radial function vanishing on $\overline{B(0,1)}$ and identically equal to $1$ on $\mathbb{R}^{n+1}\setminus \overline{B(0,2)}$. For any $\rho>0$ and $z\in\mathbb{R}^{n+1}$, denote $\varphi_\rho(z)=\varphi(z/\rho)$ and $\psi_\rho=1-\varphi_\rho$. Set
\begin{equation*}
    \tilde{\mathcal{R}}_\rho\omega^p_\Omega(z)=\int \dfrac{z-y}{|z-y|^{n+1}}\varphi_\rho(z-y)\,d\omega^p_\Omega(y).
\end{equation*}
Note that
\begin{flalign*}
    |\tilde{\mathcal{R}}_r\omega^p_\Omega(x)-\mathcal{R}_r\omega^p_\Omega(x)|&\leq\int_{r<|x-y|<2r}\dfrac{d\omega^p_\Omega(y)}{|x-y|^n}\leq 2^n\Theta^n_{\omega^p_\Omega}(x,2r) \\
    &\leq 2^n m\Theta^n_{\omega^p_\Omega}(x,r(B_R))\leq 2^n C_0 m\Theta^n_{\omega^p_\Omega}(B_R),
\end{flalign*}
and the same bound holds for $|\tilde{\mathcal{R}}_{A_0^{-1}r(B_R)}\omega^p_\Omega(x)-\mathcal{R}_{A_0^{-1}r(B_R)}\omega^p_\Omega(x)|$. Consequently,
\begin{equation*}
    I\leq|\tilde{\mathcal{R}}_r\omega^p_\Omega(x)-\tilde{\mathcal{R}}_{A_0^{-1}r(B_R)}\omega^p_\Omega(x)|+2^{n+1}C_0 m\Theta^n_{\omega^p_\Omega}(B_R).
\end{equation*}
For any $\rho>0$ and $z\in \mathbb{R}^{n+1}\setminus[\operatorname{supp}(\varphi_\rho(x-\cdot)\omega^p_\Omega)\cup\{p\}]$, define the harmonic function
\begin{equation}\label{eq:green_agujero_suave}
    u_\rho(z)=\mathcal{E}(z-p)-\int\mathcal{E}(z-y)\varphi_\rho(x-y)\,d\omega^p_\Omega(y).
\end{equation}
By \eqref{eq:Green_func}, we can write
\begin{equation}\label{eq:funcion_green_rellenando_agujeros}
    G^p_\Omega(z)=u_\rho(z)-\int\mathcal{E}(z-y)\psi_\rho(x-y)\,d\omega^p_\Omega(y)\quad\text{for }\mathcal{L}^{n+1}\text{-a.e. } z\in\mathbb{R}^{n+1}.
\end{equation}
Differentiating \eqref{eq:green_agujero_suave} with respect to $z$, evaluating at $z=x$, and using the fact that the point-mass term $\nabla \mathcal{E}(x-p)$ cancels out when taking differences, we obtain
\begin{flalign*}
    |\tilde{\mathcal{R}}_r\omega^p_\Omega(x)-\tilde{\mathcal{R}}_{A_0^{-1}r(B_R)}\omega^p_\Omega(x)| &= \kappa_{n+1}|\nabla u_r(x)-\nabla u_{A_0^{-1}r(B_R)}(x)| \\
    &\lesssim_n |\nabla u_r(x)|+|\nabla u_{A_0^{-1}r(B_R)}(x)|.
\end{flalign*}
Using \eqref{eq:funcion_green_rellenando_agujeros} and the harmonicity of $u_r$ in $B_r(x)$, we estimate
\begin{flalign}\label{eq:split}
    |\nabla u_r(x)|&\lesssim_n\frac{1}{r\mathcal{L}^{n+1}(B_r(x))}\int_{B_r(x)}|u_r(z)|\,dz \notag\\
    &\leq\frac{1}{r}\dashint_{B_r(x)}\left|\int\mathcal{E}(z-y)\psi_r(x-y)\,d\omega^p_\Omega(y)\right|dz + \frac{1}{r}\dashint_{B_r(x)}G_\Omega^p(z)\,dz \notag\\
    &=: I_1 + I_2.
\end{flalign}
For $I_1$, using Fubini's theorem and $\operatorname{supp}\psi_r(x-\cdot)\subset B(x,2r)$, we get
\begin{flalign*}
    I_1&\lesssim_n\frac{1}{r^{n+2}}\int\int_{B_r(x)}\mathcal{E}(z-y)\psi_r(x-y)\,dz\,d\omega_\Omega^p(y) \\
    &\lesssim_n\frac{1}{r^{n+2}}\int_{B_{2r}(x)}\int_{B_{3r}(y)}\dfrac{dz}{|y-z|^{n-1}}\,d\omega^p_\Omega(y)\lesssim_n\Theta^n_{\omega^p_\Omega}(x,2r)\leq C_0 m\Theta^n_{\omega^p_\Omega}(B_R).
\end{flalign*}
To bound $I_2$, recall $Q_1\in\mathcal{D}_{\omega^p_\Omega}^{db}$. By Lemmas \ref{lema:Cap_8B} and \ref{lema:densidad_acotada_cdc_local},
\begin{equation*}
    G^p_\Omega(z)\lesssim_n\dfrac{\omega^p_\Omega(8\overline{B_{Q_1}})}{\Capp(2\overline{B_{Q_1}}\setminus\Omega)}\lesssim_{n,m,C_0}\dfrac{\omega^p_\Omega(B_{Q_1})}{r(B_{Q_1})^{n-1}}\quad\text{for all } z\in \overline{B_{Q_1}}\cap\Omega.
\end{equation*}
Since $r\leq r(Q_1)\leq C_0 A_0 r$ and \eqref{eq:inclusiones_bolas_riesz_DM} holds, we deduce
\begin{equation*}
    I_2\lesssim_{n,m,C_0}\dfrac{\omega^p_\Omega(B_{Q_1})}{r(B_{Q_1})^n}.
\end{equation*}
Note that $x\in B_{Q_1}\subset B(x,58r(Q_1))$ and $58r(Q_1)\leq 58C_0 A_0^{-j_0}<\beta$ by \ref{enum:beta}. If $58r(Q_1)\leq r(B_R)$, then
\begin{equation*}
    \Theta^n_{\omega^p_\Omega}(B_{Q_1})\leq\Theta^n_{\omega^p_\Omega}(x,58r(Q_1))\leq m \Theta^n_{\omega^p_\Omega}(x,r(B_R))\leq C_0 m\Theta^n_{\omega^p_\Omega}(B_R).
\end{equation*}
In the case $58r(Q_1)>r(B_R)$, it holds that $B_{Q_1}\subset 59C_0 B(R)$. Indeed, letting $x_R$ be the center of $B(R)$ and $y\in B_{Q_1}$, since $r(Q_1)\leq C_0 A_0^{-j_0}\leq C_0 r(R)$ and $x\in 2B_R$, we have
\begin{equation*}
    |x_R-y|\leq|x_R-x|+|x-y|\leq 58r(R)+58r(Q_1)\leq 59C_0 r(R).
\end{equation*}
Thus, using the $(C,P_{\omega^p_\Omega,\gamma})$-doubling properties of $B(R)$, we conclude that
\begin{equation*}
    \Theta^n_{\omega^p_\Omega}(B_{Q_1})\lesssim_n \dfrac{\omega^p_{\Omega}(59C_0 B(R))}{r(B_R)^n}\lesssim_{n,\gamma,a,C_0}\Theta^n_{\omega^p_\Omega}(B_R).
\end{equation*}
Replacing $r$ by $A_0^{-1}r(B_R)$ in \eqref{eq:split}, we obtain the analogous bound $|\nabla u_{A_0^{-1}r(B_R)}(x)|\leq J_1+J_2$. The term $J_1$ is bounded exactly as $I_1$. For $J_2$, since $\tilde{R}\in\mathcal{N}(R)$ is the top cube, $\tilde{R}\in\mathcal{D}_{\omega^p_\Omega}^{db}$ and $\tilde{R}\notin\mathcal{B}_{\tilde{R}}$, we apply Lemmas \ref{lema:Cap_8B} and \ref{lema:densidad_acotada_cdc_local} over $\tilde{R}$ to conclude $J_1+J_2\lesssim_{n,\gamma,a,C_0}\Theta^n_{\omega^p_\Omega}(B_R)$.

\medskip
\noindent\textbf{Case 2:} $Q_1\notin\mathcal{D}_{\omega^p_\Omega}^{db}$, but $\{Q_l\}_{l=1}^N\cap \mathcal{D}_{\omega^p_\Omega}^{db}\neq\emptyset$. Define
\begin{equation*}
    l_0=\min\left\{1\leq l\leq N:Q_l\in\mathcal{D}_{\omega^p_\Omega}^{db}\right\},
\end{equation*}
so $Q_{l_0}$ is the smallest doubling cube in this chain. We decompose
\begin{equation*}
    I\leq |\mathcal{R}_r\omega^p_{\Omega}(x)-\mathcal{R}_{A_0^{l_0-1}r}\omega^p_{\Omega}(x)|+|\mathcal{R}_{A_0^{l_0-1}r}\omega^p_{\Omega}(x)-\mathcal{R}_{A_0^{-1}r(B_R)}\omega^p_{\Omega}(x)|.
\end{equation*}
The term $|\mathcal{R}_{A_0^{l_0-1}r}\omega^p_{\Omega}(x)-\mathcal{R}_{A_0^{-1}r(B_R)}\omega^p_{\Omega}(x)|$ is bounded as in Case 1. For the first term, since $Q_l\notin\mathcal{D}_{\omega^p_\Omega}^{db}$ for all $1\leq l\leq l_0-1$, applying Lemma \ref{lema:acotacion_cadena_cubos_malos} with $\kappa$ chosen larger than $n+1$ in \eqref{eq:C_0^-l_0_decrece_mas_que_A_0} and using \eqref{eq:inclusiones_bolas_riesz_DM}, we get
\begin{flalign*}
    |\mathcal{R}_r\omega_\Omega^p(x)-\mathcal{R}_{A_0^{l_0-1}r}\omega_\Omega^p(x)| &= \left|\int_{A_0^{l_0-1}r\geq|x-y|>r}\dfrac{x-y}{|x-y|^{n+1}}\,d\omega^p_\Omega(y)\right| \\
    &\leq\sum_{j=1}^{l_0-1}\int_{A_0^j r\geq|x-y|>A_0^{j-1}r}\dfrac{d\omega^p_\Omega(y)}{|x-y|^n}\leq \sum_{j=1}^{l_0-1}\dfrac{\omega^p_\Omega(B_{Q_{j+1}})}{(A_0^{j-1}r)^n} \\
    &\leq\dfrac{\omega^p_\Omega(B_{Q_{l_0}})}{(A_0^{l_0-2}r)^n}+\sum_{j=1}^{l_0-2}\dfrac{A_0^{-(n+1)(l_0-j-2)}\omega^p_\Omega(800B(Q_{l_0}))}{(A_0^{j-1}r)^n} \\
    &=\dfrac{\omega^p_\Omega(B_{Q_{l_0}})}{(A_0^{l_0-2}r)^n}+\dfrac{\omega^p_\Omega(800B(Q_{l_0}))}{(A_0^{l_0-1}r)^n}A_0^{2n}\sum_{j=0}^{l_0-3}A_0^{-j} \\
    &\lesssim_{C_0,n}\Theta^n_{\omega^p_\Omega}(B_{Q_{l_0}})\lesssim_{C_0,n,m}\Theta^n_{\omega^p_\Omega}(B_R).
\end{flalign*}

\medskip
\noindent\textbf{Case 3:} $\{Q_l\}_{l=1}^N\cap \mathcal{D}_{\omega^p_\Omega}^{db}=\emptyset$. 
For $x\in G_B\cap 2B_R$, we decompose
\begin{equation*}
    I\leq |\mathcal{R}_r\omega^p_{\Omega}(x)-\mathcal{R}_{A_0^{N-1}r}\omega^p_{\Omega}(x)|+|\mathcal{R}_{A_0^{N-1}r}\omega^p_{\Omega}(x)-\mathcal{R}_{A_0^{-1}r(B_R)}\omega^p_{\Omega}(x)|.
\end{equation*}
The first term is bounded by $\Theta^n_{\omega^p_\Omega}(B_R)$ following the non-doubling sum argument of Case 2. Recalling $N=i_0-j_0$, we have $A_0^{N-1}r\approx_{C_0} r(B_R)$, so the second term is bounded using the doubling cube $R$ similarly as we did in Case 1. Indeed, writing
$\rho_1= \min(A_0^{-1}r(B_R), A_0^{N-1}r)$, $\rho_2= \max(A_0^{-1}r(B_R), A_0^{N-1}r)$, we get
\begin{align*}
|\mathcal{R}_{A_0^{N-1}r}\omega^p_{\Omega}(x)-\mathcal{R}_{A_0^{-1}r(B_R)}\omega^p_{\Omega}(x)| &
    \leq \int_{\rho_1<|x-y|\leq \rho_2} \frac1{|x-y|^n}\,d\omega^p_{\Omega}(y) \\
    &\leq 
    \frac{\omega_\Omega^p(B(x,\rho_2))}{\rho_1^n} \leq C\,\Theta^n_{\omega^p_\Omega}(B_R).
    \end{align*}

\vspace{2mm}

Before checking assumption b) of Theorem \ref{teo:main1} (which is the most delicate part of our proof), we will deal with hypothesis c), using arguments similar to those used for a).

\subsection{Hypothesis c)}
To show this assumption, we may need the explicit expression for $\mathcal{R}\omega^p_\Omega(x)$. Let us show that
    \begin{equation}\label{eq:explicit_formula_riesz}
        \mathcal{R}\omega^p_\Omega(x)=\dfrac{x-p}{|x-p|^{n+1}}\quad\omega^p_\Omega\text{-a.e. }x\in E_{m,\beta}.
    \end{equation}
    As we explained in the beginning of the proof, $\omega^p_\Omega$-a.e. $x\in E_{m,\beta}$ is a density point of $E_{m,\beta}$ for $\omega^p_\Omega$, and those points have a sequence of $(a,P_{\gamma,\omega^p_\Omega})$-doubling balls $\{B(x,r_j)\}_{j\in\mathbb{N}}$ with $r_j\to0$ when $j\to\infty$. Once again, since 
    \begin{equation*}
        \lim_{j\to\infty}\dfrac{\omega^p_\Omega(B(x,r_j)\cap E_{m,\beta})}{\omega^p_\Omega(B(x,r_j))}=1,
    \end{equation*}
    we can take $j_0\in\mathbb{N}$ large enough and argue as in Lemma \ref{lema:densidad_acotada_cdc_local} to show that\\
    $\Capp(\overline{B(x,r_{j})}\setminus\Omega)\gtrsim_{n,m,a}r_{j}^{n-1}$ for every $j\geq j_0$. Now, if we repeat the argument in the proof of \textit{Hypothesis a)} with $B(x,r_j)$ instead of $B_R$, for all $0<r\lesssim r(R)$, we can get the following estimate
    \begin{flalign*}
        \left|\mathcal{R}_{r}\omega^p_\Omega(x)-\frac{x-p}{|x-p|^{n+1}}\right|&\leq \left|\tilde{\mathcal{R}}_{r}\omega^p_\Omega(x)-\frac{x-p}{|x-p|^{n+1}}\right|+2^n\Theta^n_{\omega^p_\Omega}(x,2r)\\
        &\lesssim_{C_0,\gamma,a,n,m}\Theta^n_{\omega^p_\Omega}(x,r_j).
    \end{flalign*}
    Since $E_{m,\beta}\subset Z$, $\Theta^n_{\omega^p_\Omega}(x,r_j)\to0$ when $j\to\infty$, the identity \eqref{eq:explicit_formula_riesz} follows. From this and the fact that $\dfrac{x-p}{|x-p|^{n+1}}$ is a Calder\'on--Zygmund kernel, we can obtain
    \begin{equation}\label{eq:estimacion_integral_riesz}
        \begin{aligned}
             \int_{2B_R\cap G_B}|\mathcal{R}\omega^p_\Omega(x)-m_{\omega^p_\Omega,2B_R\cap G_B}&(\mathcal{R}\omega^p_\Omega)|^2d\omega^p_\Omega(x)\\
             &\lesssim_n\omega^p_\Omega(2B_R\cap G_B)\left( \dfrac{r(R)}{|x_R-p|^{n+1}}\right)^2,
        \end{aligned}
    \end{equation}
    where $x_R$ is the center of $B(R)$. In order to bound this, observe that
    \begin{equation*}
        \dfrac{r(R)}{|x_R-p|^{n+1}}\leq \dfrac{r(R)}{d(\partial\Omega,p)^{n+1}}\leq k(p,n,\Omega)\dfrac{r(R)}{\operatorname{diam}(\partial\Omega)^{n+1}}\omega^p_\Omega(\partial\Omega).
    \end{equation*}
From here, we deduce
\begin{equation*}
    \begin{aligned}
         \dfrac{r(R)}{\operatorname{diam}(\partial\Omega)^{n+1}}&\omega^p_\Omega(\partial\Omega)\\
         &=\left(\dfrac{r(R)}{\operatorname{diam}(\partial\Omega)}\right)^\gamma\dfrac{\omega^p_\Omega(B(x_R,\operatorname{diam}(\partial\Omega)))}{\operatorname{diam}(\partial\Omega)^{n}}\left(\dfrac{r(R)}{\operatorname{diam}(\partial\Omega)}\right)^{1-\gamma}\\
         &\leq 2^{n+\gamma}P_{\omega^p_\Omega}(B(R))\left(\dfrac{r(R)}{\operatorname{diam}(\partial\Omega)}\right)^{1-\gamma}\\
         &\leq2^{n+\gamma} a\Theta^n_{\omega^p_\Omega}(B(R))\left(\dfrac{r(R)}{\operatorname{diam}(\partial\Omega)}\right)^{1-\gamma}.\\
    \end{aligned}
\end{equation*}
This and \eqref{eq:estimacion_integral_riesz} imply that
\begin{equation*}
    \begin{aligned}
        \int_{2B_R\cap G_B}|\mathcal{R}\omega^p_\Omega(x)-&m_{\omega^p_\Omega,2B_R\cap G_B}(\mathcal{R}\omega^p_\Omega)|^2d\omega^p_\Omega(x)\\
        &\lesssim_{n,p,\gamma,\Omega}\omega_\Omega^p(B_R)\Theta^n_{\omega^p_\Omega}(B_R)^2\left(\dfrac{r(R)}{\operatorname{diam}(\partial\Omega)}\right)^{2-2\gamma}.
    \end{aligned}
\end{equation*}
    Now, taking $r(R)$ small enough (that is, $r_{k_0}$ small enough), then the assumption $c)$ is fulfilled.

\subsection{Hypothesis b)}
Throughout this part of the proof, $\tilde{R}$ denotes a cube in $\mathcal{N}(R)$. Let $k\in\mathbb{N}$ be an integer multiple of $4$, sufficiently large, to be chosen later depending only on $C_0,n,a,\gamma$ and $m$. Define
\begin{flalign*}
    \mathcal{G}_0(\tilde{R})=\{&Q\in \mathcal{D}_{\omega_\Omega^p,k}(\tilde{R})~:\text{ at least }k/2\text{ cubes }~Q\subsetneqq S\subsetneqq \tilde{R}~\text{ satisfy } S\in\mathcal{D}_{\omega^p_\Omega}^{db}~\text{ and }~\\
    &\omega_\Omega^p(S\setminus E_{m,\beta})<(\delta')^{1/2}\omega_\Omega^p(S)\},
\end{flalign*}
where $\delta'$ is the constant appearing in Lemma \ref{lema:G_B_mucha_masa_dentro_de_3B}. Also, denote
\begin{equation*}
    \mathcal{G}_0(R')=\bigcup_{\tilde{R}\in\mathcal{N}(R)}\mathcal{G}_0(\tilde{R}).
\end{equation*}
For every $Q\in\mathcal{G}_0(R')$, we can consider the intermediate cube $Q\subsetneqq P_Q\subsetneqq \tilde{R}$ such that $P_Q\in\mathcal{D}_{\omega^p_\Omega}^{db}$, $\omega_\Omega^p(P_Q\setminus E_{m,\beta})<(\delta')^{1/2}\omega_\Omega^p(P_Q)$, and has exactly $k/4$ predecessors in $\mathcal{D}_{\omega_\Omega^p}(\tilde{R})\setminus\{\tilde{R}\}$ also satisfying both conditions. Thus $P_Q$ is maximal among the cubes satisfying these properties. Observe that such cubes satisfy the assumptions of Lemma \ref{lema:densidad_acotada_cdc_local} if we take $\delta'<\frac{1}{4C_0^2}$. We set
\begin{equation}\label{eqdefGR'}
    \mathcal{G}(R')=\{P_Q:Q\in\mathcal{G}_0(R')\}.
\end{equation}

Let $\varphi:\mathbb{R}^{n+1}\rightarrow[0,1]$ be a radial $\mathcal{C}^\infty$ function such that $\varphi|_{B(R)}=1$, $\operatorname{supp}\varphi\subset 2B(R)$ and $\|\nabla^2\varphi \|_\infty\lesssim_nr(B_R)^{-2}$. Now, 
since $p\not\in 
2B(R)$ by \ref{enum:pol_p}, and by Lemma \ref{lema:formula_Green_medida_armonica}, we obtain
\begin{flalign*}
    \omega_\Omega^p\left(B(R)\right)&\leq\int\varphi(y)\,d\omega^p(y)\\
    &=\int_{2B(R)}\Delta\varphi(y)G_\Omega^p(y)\,dy\lesssim_n\sup_{2B(R)}G^p_\Omega(\cdot) r(B_R)^{n-1}.
\end{flalign*}
So taking $q\in 2B(R)$ such that $G_\Omega^p(q)\geq\sup_{2B(R)}G^p_\Omega(\cdot) /2$, we have
\begin{equation}\label{eq:punto_q_en_B_R}
    G_\Omega^p(q)\gtrsim_n\dfrac{\omega_\Omega^p(B(R))}{r(B_R)^{n-1}}\gtrsim_{C_0}\dfrac{\omega_\Omega^p(B_R)}{r(B_R)^{n-1}}.
\end{equation}
The goal is to use a touching point argument with a ball centered in $q$.

First we remark that the touching ball $\overline{B(q,d(q,\partial\Omega))}$ intersects $\partial\Omega$ in a point of $R'$. In fact, since $q\in 2B(R)$, we have
\begin{equation*}
\overline{B(q,d(q,\partial\Omega))}\cap\partial\Omega\subset B(q,3r(B))\cap\partial\Omega\subset B_R\cap\partial\Omega\subset R'.
\end{equation*}
We would also like the touching ball to intersect $\bigcup_{P\in\mathcal{G}(R')}P$.
To overcome this difficulty, we will modify the domain by removing some 'bad' cubes from the boundary.

We first show that the point $q$ has to be 'away' from the cubes in $\mathcal{G}(R')$ if the integer $k$ that defines the scale of the cubes in $\mathcal{G}_0(R')$ is large enough. Aiming for a contradiction, assume that 
\begin{equation}\label{eq:contradiccion_q_cerca_cubos_buenos}
    q\in 4B_{P_Q} \quad\text{for some}\quad Q\in\mathcal{G}_0(R').
\end{equation}
We know that there exists a sequence of cubes
\begin{equation*}
    P_0=P_Q\subsetneqq P_1\subsetneqq\cdots\subsetneqq P_N=\tilde{R}\quad\text{in}\quad\mathcal{D}_{\omega_\Omega^p}(\tilde{R}).
\end{equation*}
Denoting by $B_{P_i}$ the balls associated to the cubes $P_i$, by definition we know that there are exactly $k/4$ intermediate cubes $P_i$ with $0\leq i\leq N-1$ such that $\omega_\Omega^p(P_i\setminus E_{m,\beta})<(\delta')^{1/2}\omega^p(P_i)$ and $P_i\in\mathcal{D}^{db}_{\omega_\Omega^p}$. Define the open sets
\begin{equation*}
    \Omega_i=5B_{P_i}\cap\Omega\quad\text{for every}\quad0\leq i\leq N.
\end{equation*} 
Observe that each $\Omega_i$ is Wiener regular, $p\not\in570C_0B(R)$ by \ref{enum:pol_p}, and $G_{\Omega}^p|_{\Omega_i}$ is harmonic in $\Omega_i$ and vanishes as $x\to\xi\in\partial\Omega\cap5B_{P_i}$. Now, for every $x\in \partial (5B_{P_i})\cap\Omega$

\begin{equation}\label{eq:primera_estimacion_contradiccion_q}
\begin{aligned}
    G^p_{\Omega}(x)&=\int_{\partial\Omega_{i+1}}G^p_{\Omega}(y)\,d\omega^x_{\Omega_{i+1}}(y)=\int_{\partial (5B_{P_{i+1}})\cap\Omega}G^p_{\Omega}(y)\,d\omega^x_{\Omega_{i+1}}(y)\\
    &\leq\omega^x_{\Omega_{i+1}}(\partial (5B_{P_{i+1}})\cap\Omega)\sup_{\partial (5B_{P_{i+1}})\cap\Omega}G^p_{\Omega}(\cdot).
\end{aligned}
\end{equation}
To apply the estimates from Section 2, we need to verify a geometric inclusion. This is the purpose of the following lemma.

\begin{lema}\label{lema:5B_P_dentro_de_0.99_5B}
    $5B_{P_i}\subset0.99\frac{5}{4}B_{P_{i+1}}$ for every $0\leq i\leq N-1$.
\end{lema}
\begin{proof}
    Recall from \eqref{eq:control_radios_DM} that $29A_0^{-{k(i)}}\leq r(B_{P_i})\leq 29C_0A_0^{-{k(i)}}$, by \eqref{eq:cadena_inclsuion_cubos_bolas_DM} $\partial\Omega\cap B(P_i)\subset P_i\subset\partial\Omega \cap28B(P_i)$ and also $A_0>5000C_0$. Since $P_{i}\subset P_{i+1}$, if $x_i$ and $x_{i+1}$ are the centers of $B(P_i)$ and $B(P_{i+1})$ respectively, then for every $y\in5B_{P_i}$
    \begin{flalign*}
        &| x_{i+1}-y |\leq | x_{i+1}-x_i |+| x_{i}-y |<28r(P_{i+1})+5r(B_{P_i})\\
        & <28r(P_{i+1})+145C_0A_0^{-k(i)}<28r(P_{i+1})+\dfrac{145}{5000}A_0^{-k(i)+1}\\
        &<\left(28+\dfrac{145}{5000}\right)r(P_{i+1})<0.99\dfrac{5}{4}(29r(P_{i+1}))=0.99\dfrac{5}{4}r(B_{P_{i+1}}).
    \end{flalign*}
\end{proof}
Using that $\omega_{\Omega_i}$ is a probability measure and \eqref{eq:primera_estimacion_contradiccion_q}, we get the trivial estimate
\begin{equation*}
    \sup_{\partial (5B_{P_{i}})\cap\Omega}G^p_{\Omega}(\cdot)\leq \sup_{\partial (5B_{P_{i+1}})\cap\Omega}G^p_{\Omega}(\cdot),\quad\text{for every }~0\leq i\leq N-1.
\end{equation*}
We now improve this inequality. Assume that $P_i\in\mathcal{D}^{db}_{\omega^p}$ and $\omega_\Omega^p(P_i\setminus E_{m,\beta})<(\delta')^{1/2}\omega^p(P_i)$. Then, since $\Omega_{i+1}$ is bounded and $x\in\partial (5B_{P_i})\cap\Omega$, we get
\begin{equation*}
    \begin{aligned}
        \omega_{\Omega_{i+1}}^x&(\partial (5B_{P_{i+1}})\cap\Omega)=1-\omega_{\Omega_{i+1}}^x\left( 5\overline{B_{P_{i+1}}}\cap\partial\Omega\right)\\
        &\leq1-\omega_{\Omega_{i+1}}^x\left( 0.99(5\overline{B_{P_{i+1}}})\right)\overset{\text{L.}\ref{lema:5B_P_dentro_de_0.99_5B}~\&~ \ref{lema:cota_superior_medida_armonica_bolas}}{\leq}
        1-c(n)\dfrac{\Capp\left( 0.99\frac{5}{4}\overline{B_{P_{i+1}}}\setminus\Omega_{i+1} \right)}{r(B_{P_{i+1}})^{n-1}}\\
        &=1-c(n)\dfrac{\Capp\left( 0.99\frac{5}{4}\overline{B_{P_{i+1}}}\setminus\Omega \right)}{r(B_{P_{i+1}})^{n-1}}\overset{\text{L.} \ref{lema:5B_P_dentro_de_0.99_5B}}{\leq}1-c(n)\dfrac{\Capp\left( \overline{B_{P_{i}}}\setminus\Omega \right)}{r(B_{P_{i+1}})^{n-1}}\\
        &\overset{\text{L.}\ref{lema:densidad_acotada_cdc_local}}{\leq}1-\hat{c}(n,m,C_0)\dfrac{r(B_{P_i})^{n-1}}{r(B_{P_{i+1}})^{n-1}}\overset{\eqref{eq:control_radios_DM}}{\leq}1-c(n,m,C_0)=\lambda(n,m,C_0)<1.
    \end{aligned}
\end{equation*}
Combining this estimate with \eqref{eq:primera_estimacion_contradiccion_q}, we deduce that, for every cube such that $P_i\in\mathcal{D}^{db}_{\omega^p}$ and $\omega_\Omega^p(P_i\setminus E_{m,\beta})<(\delta')^{1/2}\omega^p(P_i)$, 
    \begin{equation*}
        \sup_{\partial (5B_{P_{i}})\cap\Omega}G^p_{\Omega}(\cdot)\leq\lambda \sup_{\partial (5B_{P_{i+1}})\cap\Omega}G^p_{\Omega}(\cdot).
    \end{equation*}
    Since there are $k/4$ cubes $P_i$ as above, by the maximum principle and the fact that 
    $5B_{P_{N-1}} \subset B_{\tilde{R}}$,     
    we get
    \begin{equation}\label{eq:contradiccion_k_grande}
        G^p_\Omega(q)\leq\sup_{\partial (5B_{P_0})\cap\Omega}G^p_{\Omega}(\cdot)\leq\lambda^{k/4} \sup_{\partial( 5B_{P_{N-1}})\cap\Omega}G^p_{\Omega}(\cdot)\leq\lambda^{k/4} \sup_{\partial B_{\tilde{R}}\cap\Omega}G^p_{\Omega}(\cdot).
    \end{equation}
    Now, we can argue as in \eqref{eq:800B(R)_tiene_masa_en_Emb} and Lemma \ref{lema:densidad_acotada_cdc_local} to obtain
    \begin{equation}\label{eq:cota_inferior_capacidad_bola_grande}
        \Capp(118C_0\overline{B(R)}\setminus\Omega)\gtrsim_{n,m,C_0,a,\gamma}r(B_R)^{n-1}.
    \end{equation}
    Since $B_{\tilde{R}}\subset 59C_0B(R)$, by \eqref{eq:cota_inferior_capacidad_bola_grande}, the maximum principle, Lemma \ref{lema:Cap_8B}, and the $(C,P_{\omega^p_\Omega})$-doubling property of $R$, we get 
    \begin{equation*}
        \sup_{\partial B_{\tilde{R}}\cap\Omega}G^p_{\Omega}(\cdot)\leq \sup_{\partial (59C_0B(R))\cap\Omega}G^p_{\Omega}(\cdot)\lesssim_{n,m,C_0,a,\gamma}\dfrac{\omega^p_\Omega(B_R)}{r(B_R)^{n-1}},
    \end{equation*}
    from this, \eqref{eq:punto_q_en_B_R}, and \eqref{eq:contradiccion_k_grande}, we conclude
    \begin{equation*}
        0<G^p_\Omega(q)\lesssim_{n,m,C_0,a,\gamma}\lambda^{k/4} G^p_\Omega(q).
    \end{equation*}
    Choosing $k=k(n,m,C_0,a,\gamma)$ sufficiently large contradicts this inequality, so \eqref{eq:contradiccion_q_cerca_cubos_buenos} cannot hold.
       Fixing that $k$  sufficiently large, then we have
    \begin{equation}\label{eq:distancia_q_cubos_buenos}
        d(q,P)\geq d(q,B_{P})\geq 3r(B_{P})\quad\text{for every}\quad P\in\mathcal{G}(R').
    \end{equation}
    
    Next we intend to show that the measure $\omega_\Omega^p(\bigcup_{Q\in\mathcal{G}_0(R')}Q)$ is large, in some sense that will be more precise below. To this end, first we claim
    \begin{equation}\label{eq:cardinal_vecinos}
        \# \mathcal{N}(R)\leq (143C_0)^{n+1}.
    \end{equation}
        Recall that $\tilde{R}\in\mathcal{N}(R)$ if $\tilde{R}\in\mathcal{D}_{\omega^p_\Omega,j_0}$ and $\tilde{R}\cap 3B_R\neq\emptyset$, so for every $y\in B(\tilde{R})$, if $y_R\in \tilde{R}\cap 3B_R$, then
        \begin{equation*}
            |y-x_R|\leq|y-y_R|+|y_R-x_R|\leq 56r(\tilde{R})+87r(R)\leq143C_0A_0^{-j_0}.
        \end{equation*}
        Hence $B(\tilde{R})\subset B(x_R,143C_0A_0^{-j_0})$, arguing as in Lemma \ref{lema:cardinal_childs_en_DM}, we obtain \eqref{eq:cardinal_vecinos}.

      Now, for each $\tilde{R}\in\mathcal{N}(R)$, we set
       \begin{equation*}
            \operatorname{Bad}_1(\tilde{R})=\{Q\in\mathcal{D}_{\omega_\Omega^p,k}(\tilde{R}):\text{at least }k/4\text{ cubes }Q\subsetneqq P\subsetneqq \tilde{R}\text{ satisfy } P\not\in\mathcal{D}^{db}_{\omega_\Omega^p}\}
        \end{equation*}
        and
        \begin{flalign*}
            \operatorname{Bad}_2(\tilde{R})=\{&Q\in\mathcal{D}_{\omega_\Omega^p,k}(\tilde{R}):\text{at least }k/4\text{ cubes }Q\subsetneqq P\subsetneqq \tilde{R}\text{ satisfies }\\
            &\omega_\Omega^p(P\setminus E_{m,\beta})\geq(\delta')^{1/2}\omega_\Omega^p(P)\}.
        \end{flalign*}
We also denote
 \begin{equation*}
            \operatorname{Bad}_1(R')=\bigcup_{\tilde{R}\in\mathcal{N}(R)}\operatorname{Bad}_1(\tilde{R}),\quad \operatorname{Bad}_2(R')=\bigcup_{\tilde{R}\in\mathcal{N}(R)}\operatorname{Bad}_2(\tilde{R}),
        \end{equation*}
so that we have
        \begin{equation*}
            \operatorname{Bad}_0(R')=\mathcal{D}_{\omega_\Omega^p,k}(R')\setminus \mathcal{G}_0(R')=\operatorname{Bad}_1(R')\cup\operatorname{Bad}_2(R').
        \end{equation*}

 First we will estimate $\omega_\Omega^p(\bigcup_{Q\in\operatorname{Bad}_1(R')}Q)$. For a fixed $Q\in\operatorname{Bad}_1(R')\cap \mathcal{D}_{\omega_\Omega^p}(\tilde R)$ with $\tilde{R}\in\mathcal{N}(R)$ 
 and $P$ as in the above definition of
  $\operatorname{Bad}_1(\tilde{R})$, by Lemma \ref{lema:acotacion_cadena_cubos_malos} and the fact that $\omega^p_\Omega(800B(P))\leq\omega^p_\Omega(800B(P'))$ for every $P\in\operatorname{Ch}(P')$, we deduce
    \begin{equation*}
    \begin{aligned}
        \omega_{\Omega}^p(Q)\leq\omega^p_\Omega(800B(Q))\leq \left(A_0^{-\kappa}\right)^{k/4}\omega^p_\Omega(800B(\tilde{R})).
    \end{aligned}
    \end{equation*}
    One easily checks that $800B(\tilde{R})\subset 829C_0B(R)\subset 2^{10+\lceil \log_2C_0\rceil}B(R)$, and so by the $(C,P_{\omega^p_\Omega,\gamma})$-doubling property of $B(R)$ and \eqref{eq:P-doubling_implica_doubling}, we get 
    \begin{equation*}
        \omega_{\Omega}^p(Q)\leq (A_0^{-\kappa})^{k/4}2^{11(n+1)}C_0^{n+1}C\omega^p_\Omega(R)
        = A_0^{-2(n+1)k}\,2^{11(n+1)}\,C_0^{n+1}C\omega^p_\Omega(R),
    \end{equation*}
    recalling that $\kappa=8(n+1)$.

    By Lemma \ref{lema:cardinal_childs_en_DM} and \eqref{eq:cardinal_vecinos}, for $k$ large enough then we have
    \begin{equation}\label{eqbad1**}
    \begin{aligned}
        \omega^p_\Omega\left( \bigcup_{Q\in\operatorname{Bad}_1(R')}Q \right)&\leq \#\mathcal{N}(R)A_0^{-2(n+1)k}\,2^{11(n+1)}\,C_0^{n+1}C\omega^p_\Omega(R)\\
  & \leq (143C_0)^{n+1} (23C_0A_0)^{(n+1)k}A_0^{-2(n+1)k}\,2^{11(n+1)}\,C_0^{n+1}C\omega^p_\Omega(R)\\
              &\leq2^{-2}C_0^{-1}\omega^p_\Omega(R).
    \end{aligned}    
    \end{equation}
    
 Next we deal with  $\omega_\Omega^p(\bigcup_{Q\in\operatorname{Bad}_2(R')}Q)$. 
  For each $Q\in\operatorname{Bad}_2(R')\cap \mathcal{D}_{\omega_\Omega^p}(\tilde R)$ with $\tilde{R}\in\mathcal{N}(R)$, we can find at least $k/4$ intermediate cubes $P$, $Q\subsetneqq P\subsetneqq \tilde R$ such that $\omega_\Omega^p(P\setminus E_{m,\beta})\geq(\delta')^{1/2}\omega_\Omega^p(P)$.
Let $S_Q$ be the maximal cube $P$ of this form.  
Using that the family $I=\{S_Q:Q\in\operatorname{Bad}_2(R')\}$ is pairwise disjoint (by maximality) and the fact that, by \eqref{eq:R'_tiene_casi_toda_la_masa_en_E_mb}, $\omega_\Omega^p(R'\setminus E_{m,\beta})<\delta'\omega_\Omega^p(R')$, we deduce
    \begin{equation}\label{eqbad2**}
        \begin{aligned}
            \omega_\Omega^p\left( \bigcup_{Q\in\operatorname{Bad}_2(R')}Q \right)&\leq \sum_{Q\in I}\omega_\Omega^p(S_Q)\leq(\delta')^{-1/2}\sum_{Q\in I}\omega_\Omega^p(S_Q\setminus E_{m,\beta})\\
            &\leq(\delta')^{-1/2}\omega^p_\Omega(R'\setminus E_{m,\beta})\leq(\delta')^{1/2}\omega^p_\Omega(R')\\
            &\leq(\delta')^{1/2}2^{8(n+1)}C_0^{n+1}C\omega^p_\Omega(R).
        \end{aligned}
    \end{equation}
    Choosing $\delta'<(2^{8(n+1)}C_0^{n+1}C)^{-2}2^{-4}C_0^{-2}$, then we get
    \begin{equation*}
        \omega_\Omega^p\left( \bigcup_{Q\in\operatorname{Bad}_2(R')}Q \right)\leq 2^{-2}C_0^{-1}\omega_\Omega^p(R).
    \end{equation*}
 Together with the previous estimate for the family $\operatorname{Bad}_1(R')$, this gives
    \begin{equation}\label{eq:medida_cubos_bad}
        \omega_\Omega^p\left( \bigcup_{Q\in\operatorname{Bad}_0(R')}Q \right)\leq 2^{-1}C_0^{-1}\omega_\Omega^p(R).
    \end{equation}
    From this estimate and the fact that $R\in\mathcal{D}^{db}_{\omega^p_\Omega}$, we obtain
    \begin{equation}\label{eq:B(R)_contiene_cubos_no_malos}
        \omega_\Omega^p\left( B(R)\setminus\bigcup_{Q\in\operatorname{Bad}_0(R')}Q \right)>0.
    \end{equation}
    In particular we have $\omega_\Omega^p(\bigcup_{Q\in\mathcal{G}(R')}Q)\geq     
    \omega_\Omega^p(\bigcup_{Q\in\mathcal{G}_0(R')}Q)>0$ (recall that the family $\mathcal{G}(R')$ was defined in \eqref{eqdefGR'}).

    We need to introduce another family of bad cubes:    
    \begin{equation*}
        \operatorname{Bad}(R')=\operatorname{Bad}_0(R')\setminus \left\{ Q\in\mathcal{D}_{\omega^p,k}(R'):Q\subset\bigcup_{P\in\mathcal{G}(R')}P \right\}.
    \end{equation*}
    Notice that the new family $\operatorname{Bad}(R')$ is obtained by removing from $\operatorname{Bad}_0(R')$ the cubes that
  are contained in any of the cubes in $\mathcal{G}(R')$. This refinement will be essential to preserve some capacitary estimates for the new
  domain $\tilde \Omega$ that we will define now.

    We are now ready to introduce the modified domain. We set
    \begin{equation*}
        K=\overline{\partial\Omega\setminus\bigcup_{Q\in\operatorname{Bad}(R')}\overline{Q}}\quad\text{and}\quad\tilde{\Omega}=K^c,
    \end{equation*}
     so that $\tilde \Omega =
        (\partial\tilde{\Omega})^c$.
   Since we cannot ensure the new open set $\tilde{\Omega}$ to be Wiener regular, we must be careful when studying the harmonic measure for $\tilde \Omega$. For every $\varepsilon>0$ denote by $\tilde{\Omega}_\varepsilon$ the Wiener regular set associated to $\tilde{\Omega}$ and $\varepsilon$ obtained in the Lemma \ref{lema:modificacion_dominios_no_WR}, that is, 
   \begin{equation*}        \tilde\Omega_\varepsilon=\tilde\Omega\setminus\bigcup_{i\in I}\overline{B_i},
   \end{equation*}
   for a suitable family of closed balls $\overline B_i$ centered in $\partial\tilde\Omega$. For the same family of balls, we also consider the set
    \begin{equation*}
        \Omega_\varepsilon=\Omega\setminus\bigcup_{i\in I}\overline{B_i},
    \end{equation*}
    which is also open, coincides with $\Omega\setminus\overline{\bigcup_{i\in I}\overline{B_i}}$, and is Wiener regular, by the same arguments as in the proof of Lemma \ref{lema:modificacion_dominios_no_WR}. 
    Note that $\Omega\subset\tilde{\Omega}$ and $\Omega_\varepsilon\subset\tilde{\Omega}_\varepsilon$. Moreover, since the functions
    \begin{equation*}
        x\mapsto \int_{\partial\Omega_\varepsilon}\mathcal{E}(y-p)\,d\omega^x_{\Omega_\varepsilon}(y)\quad\text{and}\quad x\mapsto \int_{\partial\tilde{\Omega}_\varepsilon}\mathcal{E}(y-p)\,d\omega^x_{\tilde{\Omega}_\varepsilon}(y)
    \end{equation*}
    are harmonic in $x\in\Omega_\varepsilon$ and since the Green functions are non-negative and both $\tilde{\Omega}_\varepsilon$ and $\Omega_\varepsilon$ are Wiener regular, the maximum principle yields
    \begin{equation*}
        G^p_{\tilde{\Omega}_\varepsilon}(x)-G^p_{\Omega_\varepsilon}(x)\geq0\quad\text{for all }~x\in\Omega_\varepsilon\setminus\{p\}.
    \end{equation*}    
    In particular, this holds for $q\in 2B(R)\cap\Omega_\varepsilon\setminus\{p\}$ if $\varepsilon$ is small enough (depending on $q$). Then we have
    \begin{equation}\label{eq:aprox_Green_epsilon_en_q}
        G^p_{\tilde{\Omega}_\varepsilon}(q)\geq G^p_{\Omega_\varepsilon}(q)\overset{\varepsilon\to0}{\longrightarrow}G^p_{\Omega}(q)\overset{\eqref{eq:punto_q_en_B_R}}{\gtrsim_{n}}\dfrac{\omega_\Omega^p(B_R)}{r(B_R)^{n-1}},
    \end{equation}
    where the limit when $\varepsilon\to0$ follows by the same arguments used to prove Lemma \ref{lema:aproximacion_medida_con_dominios_WR_modificados} in \cite[Chapter 6]{HM}.

    Using that $q\in2B(R)$ together with the fact that, by construction, $\partial\tilde\Omega \cap \overline{R'}$ consists of a finite union of closures of cubes from $\mathcal{G}(R')$, 
     it follows that there exists at least one cube   $P_1\in\mathcal{G}(R')$ such that $d(q,\partial\tilde{\Omega})=d(q,P_1)$. Since  $P_1$ contains a cube $Q_1\in\mathcal{G}_0(R')$ such that $Q_1\subsetneq P_{1}$ with at least $k/2$ intermediate cubes with 'good' properties, then $C_0A_0^{-\frac{k}{4}-j_0}\geq r(B({P_{1}}))\geq A_0^{-\frac{3k}{4}-j_0}$. 
    
    Since $\tilde{\Omega}_\varepsilon$ is obtained by removing some closed balls of radius at most $\varepsilon$, there are two possibilities: either there exists some cube $P_1\in\mathcal{G}(R')$ such that
    \begin{equation*}
        d(q,\partial\tilde{\Omega}_\varepsilon)=d(q,P_1)\quad\text{ for infinitely many values of }\varepsilon\to0,
    \end{equation*}
    or there exists a ball $\overline{B_{i_\varepsilon}}$ used in the construction of $\tilde{\Omega}_\varepsilon$ centered at some point of $\overline{P_1}$ for some $P_1\in\mathcal{G}(R')$ such that
    \begin{equation*}
        d(q,\partial\tilde{\Omega}_\varepsilon)=d(q,\overline{B_{i_\varepsilon}})\quad\text{ for infinitely many values of }\varepsilon\to0.
    \end{equation*}
    Fixing some $P_1$ satisfying one of the two conditions, we define 
    \begin{equation*}
        B_1=B(x_1,r(B_{P_1}))=B_{P_1}.
    \end{equation*}
    
     Our objective is to prove that the ball $B_1$ satisfies the required properties for the Hypothesis b).
    Since $A_0^{-\frac{3k}{4}-j_0}\leq r(P_1)\leq C_0A_0^{-\frac{k}{4}-j_0}$ and $A_0^{-j_0}\leq r(R)\leq C_0A_0^{-j_0}$, we have
    \begin{equation*}
        C_0^{-1}A_0^{-\frac{3k}{4}}r(B_R)\leq r(B_1)\leq C_0A_0^{-\frac{k}{4}}r(B_R),
    \end{equation*}
    so $\delta_0(n,m,C_0,a,\gamma)=C_0^{-1}A_0^{-\frac{3k}{4}}$ and $\delta_1(n,m,C_0,a,\gamma)=C_0A_0^{-\frac{k}{4}}<1/2$. It remains to verify that $x_1\in B_R$. Recall that $q\in 2B(R)$ and \eqref{eq:B(R)_contiene_cubos_no_malos}, so 
    \begin{equation}\label{eq:cota_superior_distancia_touching_point}
        r:=d(q,\partial\tilde{\Omega}_\varepsilon)\leq d\left(q, (\partial\Omega \cap B(R))\setminus\bigcup_{Q\in\operatorname{Bad}_0(R')}Q \right)+\varepsilon<4r(R),
    \end{equation}
    this implies that
    \begin{equation*}
        |x_1-x_R|\leq|x_1-q|+|x_R-q|\leq 28\,r(P_1)+4r(R)+2r(R)<r(B_R).
    \end{equation*}

    To estimate the density ratio $\omega_\Omega^p(B_1)/{r(B_1)^n}$ from below, we will use a touching point argument.
    First, observe that by \eqref{eq:distancia_q_cubos_buenos}, for $\varepsilon>0$ small enough, we have
    \begin{equation}\label{eq:cota_inferior_distancia_touching_point}
        r=d(q,\partial\tilde{\Omega}_\varepsilon)\geq d(q,\partial\tilde{\Omega})-\varepsilon =d(q,P_1)-\varepsilon\geq d(q,B_1)-\varepsilon\geq 2r(B_1).
    \end{equation}
    Consider $\xi\in\overline{B(q,r)}\cap\partial\tilde{\Omega}_\varepsilon$ and let $q_1=\frac{\xi+q}{2}$ be mid-point between $\xi$ and $q$, so that $\overline{B(q_1,r/2)}\cap\partial\tilde{\Omega}_\varepsilon=\{\xi\}$. By rotation and translation invariance of harmonic measure, we may assume that $q_1=0$ and $\xi=(0,\dots,0,-|\xi|)$. We define
    \begin{equation*}
        \Gamma=\{ x\in\partial B(q_1,r/2)~:~x_{n+1}\geq0 \}.
    \end{equation*}
    Since $d(\Gamma,\partial B(q,r))\approx r$ and $G_{\tilde{\Omega}_\varepsilon}^p$ is harmonic in $B(q,r)\subset \tilde{\Omega}_\varepsilon\setminus\{p\}$, by a Harnack chain argument $G_{\tilde{\Omega}_\varepsilon}^p(y)\approx_n G_{\tilde{\Omega}_\varepsilon}^p(q)$ for all $y\in\Gamma$. From this and the maximum principle we deduce
    \begin{equation*}
        G_{\tilde{\Omega}_\varepsilon}^p(z)\gtrsim_n \omega^z_{B(q_1,r/2)}(\Gamma)G_{\tilde{\Omega}_\varepsilon}^p(q),\quad\text{for all }~z\in B(q_1,r/2).
    \end{equation*}
    In particular, if we take $z\in B(q_1,r/2)\cap B(\xi,r/4)$, by the explicit formula for the harmonic measure for the ball $B(q_1,r/2)$, we get
    \begin{equation*}
        \begin{aligned}
            \omega^z_{B(q_1,r/2)}(\Gamma)&\approx_n\int_\Gamma\dfrac{(r/2)^2-|z|^2}{(r/2)|y-z|^{n+1}}\,d\sigma (y)\\
            &=d(z,\partial B(q_1,r/2)) \int_\Gamma\dfrac{(r/2)+|z|}{(r/2)|y-z|^{n+1}}\,d\sigma (y)\\
            &\gtrsim_nd(z,\partial B(q_1,r/2))\dfrac{\sigma(\Gamma)}{r^{n+1}}\approx_n\dfrac{d(z,\partial B(q_1,r/2))}{r},
        \end{aligned}
    \end{equation*}
    where $\sigma$ denotes the surface measure. This implies that
    \begin{equation}\label{eq:conexion_funcion_green_distancia_frontera}
        G_{\tilde{\Omega}_\varepsilon}^p(z)\gtrsim_n\dfrac{d(z,\partial B(q_1,r/2))}{r}\,G_{\tilde{\Omega}_\varepsilon}^p(q),\quad\text{for all }~z\in B(q_1,r/2)\cap B(\xi,r/4).
    \end{equation}
    
    We claim that $B\left(\xi,r(P_1)/4\right)\subset B_1$. Indeed, for every $y\in B\left(\xi,r(P_1)/4\right)$, if $\xi\in\overline{B_{i_\varepsilon}}$ and $x_{i_\varepsilon}\in \overline{P_1}$ is the center of the ball $B_{i_\varepsilon}$, then
    \begin{equation*}
        \begin{aligned}
            |y-x_{P_1}|&\leq |y-\xi|+|\xi-x_{i_\varepsilon}|+|x_{i_\varepsilon}-x_{P_1}|<r(P_1)/4+\varepsilon+28r(P_1)\\
            & <r(P_1)/2+28r(P_1)<29r(P_1)=r(B_1).
        \end{aligned}
    \end{equation*}
    In case that $\xi\in\overline{P_1}$, we get the following
    \begin{equation*}
        |y-x_{P_1}|\leq|y-\xi|+|\xi-x_{P_1}|<r(P_1)/4+28r(P_1)<r(B_1).
    \end{equation*}
    
    By \eqref{eq:cota_inferior_distancia_touching_point}, we can find $z_0\in B(q_1,r/2)\cap B(\xi,r(P_1)/4)\subset B(q_1,r/2)\cap B(\xi,r/4)$ such that
    \begin{equation*}
        d(z_0,\partial B(q_1,r/2))\geq \dfrac{r(P_1)}{8}.
    \end{equation*}
    Since $z_0\in B_1$, by Lemma \ref{lema:Cap_8B}, \eqref{eq:conexion_funcion_green_distancia_frontera}, and \eqref{eq:cota_superior_distancia_touching_point}, we deduce that
    \begin{equation}\label{eq:medida_armonica_8B_P_1}
        \omega^p_{\tilde{\Omega}_\varepsilon}(8\overline{B_1})\gtrsim_nG^p_{\tilde{\Omega}_\varepsilon}(z_0)\Capp(2\overline{B_1}\setminus\tilde{\Omega}_\varepsilon)\gtrsim_n\dfrac{r(B(P_1))}{r(B_R)}G^p_{\tilde{\Omega}_\varepsilon}(q)\Capp(2\overline{B_1}\setminus\tilde{\Omega}_\varepsilon).
    \end{equation}
    To estimate the last expression, we need to use a lower bound for $\Capp(2\overline{B_1}\setminus\tilde{\Omega}_\varepsilon)$, but all our results from Lemma \ref{lema:densidad_acotada_cdc_local} are stated for the original domain $\Omega$, not for $\tilde{\Omega}_\varepsilon$. To solve this issue, the following geometric lemma helps.

    \begin{lema}\label{lema:comparacion_conjuntos_omega_tilde_y_omega}
    Under the assumptions above, it holds that
        \begin{equation*}
            \overline{B(P_1)}\setminus\Omega\subset2\overline{B_1}\setminus\tilde{\Omega}\subset2\overline{B_1}\setminus\tilde{\Omega}_\varepsilon.
        \end{equation*}
    \end{lema}
    
    Observe that this lemma, in combination with Lemma \ref{lema:densidad_acotada_cdc_local}, ensures that
    \begin{equation}\label{eqcapcap1}
    \Capp(2\overline{B_1}\setminus\tilde{\Omega}_\varepsilon)\geq 
    \Capp(\overline{B(P_1)}\setminus\Omega)\gtrsim_{n,m,C_0} r(P_1)^{n-1}.
    \end{equation}
    
    \begin{proof}[Proof Lemma \ref{lema:comparacion_conjuntos_omega_tilde_y_omega}]
        The second inclusion is immediate, since $\tilde{\Omega}_\varepsilon\subset\tilde{\Omega}$. For the first one, suppose that there exists some $z\in \overline{B(P_1)}\setminus\Omega$ such that $z\not\in2\overline{B_1}\setminus\tilde{\Omega}$. Since $B(P_1)\subset 2B_1$, this is equivalent to saying that $z\not\in\Omega$, and so $z\not\in\partial\tilde{\Omega}$. Thus there exists $\eta>0$ such that $B(z,\eta)\cap \left(\partial\Omega\setminus\bigcup_{Q\in\operatorname{Bad}(R')}\overline{Q}
        \right)=\emptyset$. Furthermore, recall $\Omega=\mathbb{R}^{n+1}\setminus\partial\Omega$, and so we have
        \begin{equation}\label{eq:frontera_omega_dentro_conjuntos_malos}
            \emptyset\neq B(z,\eta)\cap\partial\Omega\subset \bigcup_{Q\in\operatorname{Bad}(R')}\overline{Q}.
        \end{equation}
        Observe that by \eqref{eq:cadena_inclsuion_cubos_bolas_DM}, $z\in B(z,\eta)\cap \overline{B(P_1)}\cap \partial\Omega\subset B(z,\eta)\cap \frac{101}{100}B(P_1)\cap \partial\Omega\subset P_1$. Thus,  by \eqref{eq:frontera_omega_dentro_conjuntos_malos}, there exists a cube $S\in$ Bad$(R')$ such that $z\in \overline{S}$, which implies that $B(z,\eta)\cap \frac{101}{100}B(P_1)\cap\partial\Omega\cap S\neq\emptyset$, and in particular, $S\cap P_1\neq\emptyset$. This contradicts the fact that $S\in $Bad$(R').$
    \end{proof}

    From \eqref{eq:medida_armonica_8B_P_1} and \eqref{eqcapcap1}, we get
    \begin{equation*}
        \omega^p_{\tilde{\Omega}_\varepsilon}(8\overline{B_1})\gtrsim_{n,m,C_0}\dfrac{r(P_1)^n}{r(B_R)}\,G^p_{\Omega_\varepsilon}(q).
    \end{equation*}
    So by \eqref{eq:aprox_Green_epsilon_en_q} and Lemma \ref{lema:aproximacion_medida_con_dominios_WR_modificados}, taking $\varepsilon\to0$ we deduce that
    \begin{equation*}
        \omega^p_{\tilde{\Omega}}(8\overline{B_1})\gtrsim_{n,m,C_0}\dfrac{r(P_1)^n}{r(B_R)}G^p_{\Omega}(q)\gtrsim_n\dfrac{r(P_1)^n}{r(B_R)^n}\omega_\Omega^p(B_R),
    \end{equation*}
    or equivalently,
    \begin{equation}\label{eq:medida_armonica_8B_P_2}
     \frac{\omega^p_{\tilde{\Omega}}(8\overline{B_1})}{r(B_1)^n}\gtrsim_{n,m,C_0} \frac{\omega_\Omega^p(B_R)}{r(B_R)^n}.
     \end{equation}

   To conclude the proof of Hypothesis b), it remains to prove that $\omega_{\tilde{\Omega}}^p(8\overline{B_1})\lesssim \omega^p_{\Omega}(\overline{8B_1})$. To this end, we will need to relate $\omega_{\tilde{\Omega}}^p(8\overline{B_1})$ with $\omega^p_\Omega\left(\bigcup_{Q\in\operatorname{Bad}(R')}Q\right)$. Let $c=c(C_0,n,m)$ be the implicit constant in \eqref{eq:medida_armonica_8B_P_2}. Using that $\frac{r(P_1)^n}{r(B_R)^n}\geq29^{-n} C_0^{-n}A_0^{-\frac{3kn}{4}}$\!, we can choose  $\delta'$ small enough, depending on $k$ so that instead of \eqref{eq:medida_cubos_bad} (using \eqref{eqbad1**} and \eqref{eqbad2**}),  we get
  \begin{align*}
        \omega_\Omega^p\left( \bigcup_{Q\in\operatorname{Bad}_0(R')}Q \right)&\leq 
        C(n,C_0)\,(A_0^{-(n+1)k} + (\delta')^{1/2}) \,\omega_\Omega^p(R)\\
        &\leq 
        2^{-1}c^{-1}29^{-n} C_0^{-n}A_0^{-\frac{3kn}{4}}\omega_\Omega^p(R).
  \end{align*}  
    Then we obtain 
    \begin{equation}\label{eq:relacions_cubos_bad_medida_domino_no_WR}
        \omega^p_\Omega\left(\bigcup_{Q\in\operatorname{Bad}(R')}Q\right)
       \leq \omega_\Omega^p\left( \bigcup_{Q\in\operatorname{Bad}_0(R')}Q \right)         
        \leq\frac12\,\omega_{\tilde{\Omega}}^p(8\overline{B_1}).
    \end{equation}
    
    \begin{lema}\label{lema:acotacion_medida_bola_8B_dominio_no_WR_por_dominio_WR}
    Under the above assumptions, we have
        \begin{equation*}
            \omega_{\tilde{\Omega}}^p(8\overline{B_1})\leq\omega^p_{\Omega}(8\overline{B_1})+\omega^p_{\Omega}\left(\bigcup_{Q\in\operatorname{Bad}(R')}\overline{Q}\right).
        \end{equation*}
    \end{lema}
    \begin{proof}
        We claim that, for every $\varepsilon>0$,
        \begin{equation}\label{eq:acotacion_medida_bola_8B_dominios_modificados}        \omega_{\tilde{\Omega}_\varepsilon}^p(8\overline{B_1}\cap\partial\tilde{\Omega})\leq\omega^p_{\Omega_\varepsilon}(8\overline{B_1}\cap\partial\Omega)+\omega^p_{\Omega_\varepsilon}\left(\bigcup_{Q\in\operatorname{Bad}(R')}\overline{Q}\right),
        \end{equation}
        From this inequality and Lemma \ref{lema:aproximacion_medida_con_dominios_WR_modificados} we get the desired result.
        
        To prove \eqref{eq:acotacion_medida_bola_8B_dominios_modificados} we use the Markov type property from Lemma \ref{lema:Markov_property_harmonic_measure}:
        \begin{equation*}
            \begin{aligned}
                \omega^p_{\tilde{\Omega}_\varepsilon}(8\overline{B_1}\cap\partial\tilde{\Omega})&=\omega^p_{{\Omega}_\varepsilon}(8\overline{B_1}\cap\partial\tilde{\Omega}\cap\partial\tilde{\Omega}_\varepsilon)+\int_{\partial\Omega_\varepsilon\setminus\partial\tilde{\Omega}_\varepsilon} \omega^y_{\tilde{\Omega}_\varepsilon}(8\overline{B_1}\cap\partial\tilde{\Omega})\,d\omega_{\Omega_\varepsilon}^p(y)\\
                &\leq\omega^p_{\Omega_\varepsilon}(8\overline{B_1}\cap\partial\Omega)+\omega^p_{\Omega_\varepsilon}(\partial\Omega_\varepsilon\setminus\partial\tilde{\Omega}_\varepsilon).
            \end{aligned}
        \end{equation*}
        So, to conclude the proof of \eqref{eq:acotacion_medida_bola_8B_dominios_modificados} it suffices to show that
        \begin{equation*}               \partial\Omega_{\varepsilon}\setminus\partial\tilde{\Omega}_\varepsilon\subset\partial\Omega\setminus\partial\tilde{\Omega}\subset\bigcup_{Q\in\operatorname{Bad}(R')}\overline{Q}.
        \end{equation*}
        The inclusion $\partial\Omega\setminus\partial\tilde{\Omega}\subset\bigcup_{Q\in\operatorname{Bad}(R')}\overline{Q}$ is trivial from the definition of $\partial\tilde{\Omega}$. To show the first one, consider the open set $V=\mathbb{R}^{n+1}\setminus
        \overline{\bigcup_{i\in I_\varepsilon}\overline{B_i}}$, so that 
        \begin{equation*}
        \Omega_\varepsilon = \Omega\cap V\quad \mbox{ and }\quad\tilde\Omega_\varepsilon = \tilde\Omega\cap V.
        \end{equation*}
        Take $x\in \partial\Omega_{\varepsilon}\setminus\partial\tilde{\Omega}_\varepsilon =
        \partial(\Omega\cap V)\setminus\partial(\tilde\Omega\cap V)$.
        Since $\Omega\cap V\subset \tilde\Omega\cap V$, it holds that
        $\overline{\Omega\cap V}\subset \overline{\tilde\Omega\cap V}$. From the fact that $x\not \in \partial(\tilde\Omega\cap V)$, it follows that $x\in \tilde\Omega\cap V$. In particular, $x\in V$. Together with the fact that $x\in \partial(\Omega\cap V)$, this implies that
        $x\in\partial\Omega$. On the other hand, since $x\in\tilde\Omega$, we have that $x\not\in\partial\tilde\Omega$. Altogether, we have $x\in \partial\Omega\setminus \partial\tilde\Omega$.
     \end{proof}

    Observe that \eqref{eq:small_boundary_DM} implies that 
    \begin{equation*}
        \omega^p_\Omega\left( \bigcup_{Q\in\operatorname{Bad}(R')}\overline{Q} \right)=\omega^p_\Omega\left( \bigcup_{Q\in\operatorname{Bad}(R')}Q \right).
    \end{equation*} 
    Then, by Lemma \ref{lema:acotacion_medida_bola_8B_dominio_no_WR_por_dominio_WR} and \eqref{eq:relacions_cubos_bad_medida_domino_no_WR}, we get
    \begin{equation*}
        \omega_{\tilde{\Omega}}^p(8\overline{B_1})\leq 2\,\omega^p_{\Omega}(8\overline{B_1}),
    \end{equation*}
    which together with \eqref{eq:medida_armonica_8B_P_2} and the doubling property of $P_1$ yields
\begin{equation*}
    \frac{\omega^p_{\Omega}(B_1)}{r(B_1)^n}\approx\frac{\omega^p_{\Omega}(8\overline{B_1})}{r(B_1)^n}\gtrsim_{n,m,C_0} \frac{\omega_\Omega^p(B_R)}{r(B_R)^n}.
    \end{equation*}

\begin{proof}[Proof of Theorem \ref{teo:main_result}]
    Suppose, aiming for a contradiction, that $\omega^p_\Omega(\mathcal{O})>0$. Then we can find a sufficiently small ball $B_R$ centered at $\partial\Omega$ satisfying the assumptions of Theorem \ref{teo:main1}. This guarantees the existence of an $n$-rectifiable set $\Gamma$ and a constant $\tau>0$ such that 
    \begin{equation*}
        \omega^p_\Omega(\Gamma\cap 2B_R)\geq\tau\,\omega^p_\Omega(B_R).
    \end{equation*}
    On the other hand, by the doubling property of $B_R$, we have $\omega^p_\Omega(2B_R\setminus E_{m,\beta})<\delta'\omega^p_\Omega(2B_R)\leq C_0\delta'\omega_\Omega^p(B_R)$ for a fixed constant $C_0>0$. Thus, choosing $\delta'>0$ sufficiently small so that $C_0\delta'<\tau$, we obtain
    \begin{equation*}
        \omega^p_\Omega(\Gamma\cap E_{m,\beta}\cap 2B_R)\geq \omega^p_\Omega(\Gamma\cap 2B_R) - \omega^p_\Omega(2B_R\setminus E_{m,\beta}) \geq (\tau-C_0\delta')\omega^p_\Omega(B_R)>0.
    \end{equation*}
    Since $E_{m,\beta}\subset Z$, this yields $\omega^p_\Omega(\Gamma\cap Z)>0$, which directly contradicts the fact that $\omega^p_\Omega(\Gamma\cap Z)=0$. Consequently, $\omega^p_\Omega(\mathcal{O})=0$.
\end{proof}

\vspace{5mm}

\noindent Luis Lloret\\
Departament de Matem\`atiques\\
Universitat Aut\`onoma de Barcelona\\
08193 Bellaterra, Barcelona, Catalonia\\
\textit{E-mail}: Luis.Lloret@uab.cat

\vspace{5mm}
\noindent Xavier Tolsa\\
ICREA (Barcelona), Departament de Matem\`atiques (Universitat Aut\`onoma de Barcelona), and Centre de Recerca Matem\`atica, Bellaterra, Catalonia\\
\textit{E-mail}: Xavier.Tolsa@uab.cat\\


\begin{thebibliography}{99}

\bibitem{AMT}
{\sc J. Azzam, M. Mourgoglou, and X. Tolsa}, \textit{Mutual absolute continuity of interior and exterior harmonic measure implies rectifiability}, Comm. Pure Appl. Math. \textbf{70} (2017), no. 11, 2121--2163.

\bibitem{AMTV}
{\sc J. Azzam, M. Mourgoglou, X. Tolsa, and A. Volberg}, \textit{On a two-phase problem for harmonic measure in general domains}, Amer. J. Math. \textbf{141} (2019), no. 5, 1259--1279.

\bibitem{B}
{\sc J. Bourgain}, \textit{On the Hausdorff dimension of harmonic measures in higher dimension}, Invent. Math. \textbf{87} (1987), no. 3, 477--483.

\bibitem{DM}
{\sc G. David and P. Mattila}, \textit{Removable sets for Lipschitz harmonic functions in the plane}, Rev. Mat. Iberoam. \textbf{16} (2000), no. 1, 137--215.


\bibitem{JW}
{\sc P.W. Jones and T.H. Wolff}, \textit{Hausdorff dimension of harmonic measures in the plane}, Acta Math. \textbf{161} (1988), no. 1--2, 131--144.

\bibitem{La} {\sc N.S. Landkof}, \textit{Foundations of Modern Potential Theory}, Springer-Verlag, Berlin, 1972.


\bibitem{MAT}
{\sc P. Mattila}, \textit{Geometry of Sets and Measures in Euclidean Spaces}, Cambridge Studies in Advanced Mathematics, vol. 44, Cambridge University Press, Cambridge, 1995.

\bibitem{NTV}
{\sc F. Nazarov, S. Treil, and A. Volberg}, \textit{The $T(b)$-theorem on non-homogeneous spaces that proves a conjecture of Vitushkin}, preprint arXiv:1401.2479 (2014).

\bibitem{HM}
{\sc M. Prats and X. Tolsa}, \textit{Harmonic Measure in Euclidean Spaces}, lecture notes, available at \url{https://mat.uab.es/~xtolsa/mesuraharmonica.pdf}.

\bibitem{T}
{\sc X. Tolsa}, \textit{New criteria for the rectifiability of Radon measures in terms of Riesz transforms}, preprint arXiv:2512.14534v2.

\bibitem{T1}
{\sc X. Tolsa}, \textit{Growth estimates for Cauchy integrals of measures and rectifiability}, Geom. Funct. Anal. \textbf{17} (2007), no. 2, 605--643.

\bibitem{CZ-T}
{\sc X. Tolsa}, \textit{Analytic Capacity, the Cauchy Transform, and Non-Homogeneous Calder\'on--Zygmund Theory}, Progress in Mathematics, vol. 307, Birkh\"auser/Springer, Cham, 2014.

\bibitem{W}
{\sc T.H. Wolff}, \textit{Counterexamples with harmonic gradients in $\mathbb{R}^3$}, Essays on Fourier Analysis in Honor of Elias M. Stein (Princeton, NJ, 1991), Princeton Math. Ser., vol. 42, Princeton Univ. Press, Princeton, NJ, 1995, pp. 321--384.

\bibitem{W2}
{\sc T.H. Wolff}, \textit{Plane harmonic measures live on sets of $\sigma$-finite length}, Ark. Mat. \textbf{31} (1993), no. 1, 137--172.
\end{thebibliography}
\end{document}